\documentclass[11pt]{amsart}

\usepackage{amscd,amssymb,amsmath,graphicx,verbatim,xy}
\usepackage{wasysym}
\usepackage{hyperref}
\usepackage[TS1,OT1,T1]{fontenc}
\usepackage{lscape}
\usepackage{fullpage} 
\usepackage{pslatex} 
\usepackage{tikz}
\usepackage{enumerate}
\usetikzlibrary{shapes,arrows}
\xyoption{all}

\newtheorem{theorem}{Theorem}[section]
\newtheorem{lemma}[theorem]{Lemma}
\newtheorem{corollary}[theorem]{Corollary}
\newtheorem{proposition}[theorem]{Proposition}

\theoremstyle{definition}
\newtheorem{definition}[theorem]{Definition}

\newtheorem{example}[theorem]{Example}

\newtheorem{question}[theorem]{Question}

\theoremstyle{remark}
\newtheorem{remark}[theorem]{Remark}

\newcommand{\Dcal}{\ensuremath{\mathcal{D}}}
\newcommand{\Scal}{\ensuremath{\mathcal{S}}}
\newcommand{\Pcal}{\ensuremath{\mathcal{P}}}
\newcommand{\Xcal}{\ensuremath{\mathcal{X}}}

\newcommand{\Tcal}{\ensuremath{\mathcal{T}}}

\newcommand{\Ccal}{\ensuremath{\mathcal{C}}}

\newcommand{\Ucal}{\ensuremath{\mathcal{U}}}

\newcommand{\ra}{\rightarrow}

\numberwithin{equation}{section}

\begin{document}
\title{Universal localisations and tilting modules for finite dimensional algebras}
\author{Frederik Marks}
\address{Frederik Marks, Institute of algebra and number theory, University of Stuttgart, Pfaffenwaldring 57, D-70569 Stuttgart, Germany}
\email{marks@mathematik.uni-stuttgart.de}
\thanks{The author is grateful to Steffen Koenig, Julian K\"ulshammer and Jorge Vit\'oria for carefully reading a first version of this text and making helpful comments and suggestions. The author is supported by DFG-SPP 1489.}

\begin{abstract}
We study universal localisations, in the sense of Cohn and Schofield, for finite dimensional algebras and classify them by certain subcategories of our initial module category. A complete classification is presented in the hereditary case as well as for Nakayama algebras and local algebras. Furthermore, for hereditary algebras, we establish a correspondence between finite dimensional universal localisations and finitely generated support tilting modules. In the Nakayama case, we get a similar result using $\tau$-tilting modules, which were recently introduced by Adachi, Iyama and Reiten.

\end{abstract}
\maketitle

\tableofcontents

\section{Introduction}
In recent years, universal localisations, as introduced by Cohn (\cite{Cohn}) and Schofield (\cite{Sch}), became a useful tool in representation theory. They were studied in the context of tilting theory (\cite{AA},\cite{AS}) and with respect to finitely presented algebras (\cite{NRS}), showing that every finitely presented algebra is Morita equivalent to a universal localisation of a finite dimensional algebra. Furthermore, universal localisations turned out to be useful to construct recollements of derived module categories (\cite{ALK1},\cite{ALK2},\cite{CX1},\cite{CX2},\cite{CX3}). 
Despite these developments and new applications the concept of universal localisation still seems to be rather abstract and mysterious. Very few complete answers can be given. In the hereditary case, a classification of all universal localisations (up to equivalence) was obtained in \cite{KS} and \cite{Sch4}. In \cite{KS}, it was shown that for hereditary rings universal localisations are described by homological ring epimorphisms. These are epimorphisms in the category of rings (with unit) fulfilling a nice homological property. In general, this equivalence is well-known not to hold. On the one hand, universal localisations do not always yield homological ring epimorphims. A list of examples was constructed in \cite{NRS}. The reverse implication, on the other hand, does not hold due to an (non-noetherian) example in \cite{K}. However, the question of which (homological) ring epimorphisms for a given ring are universal localisations seems widely open (see \cite{MV} for a partial answer).

One motivation for the present work is to consider this question for a finite dimensional $\mathbb{K}$-algebra $A$, where $\mathbb{K}$ denotes an algebraically closed field. We suggest an approach, initially motivated by \cite{Sch3}, based on studying pairs of orthogonal subcategories in $A\mbox{-}mod$ (see Proposition \ref{prop uniloc}). This approach relies on the observation that for a finite dimensional algebra $A$ every universal localisation $A_\Sigma$ is given with respect to a certain set of finitely generated $A$-modules. In the hereditary case, our methods restrict to the consideration of $Ext$-orthogonal pairs, as studied in \cite{KS}. 

Another motivation for the present work is to study the interplay of universal localisations and tilting modules for finite dimensional $\mathbb{K}$-algebras. On the one hand, it is well-known that certain monomorphic universal localisations induce (possibly infinitely generated) tilting modules, as studied in \cite{AS}. In some cases, a classification of all tilting modules was obtained using this construction, e.g., for Dedekind domains (\cite{AS}) or for infinitely generated tilting modules over tame hereditary algebras (\cite{AS2}). On the other hand, a classification of universal localisations in terms of certain subcategories of the initial module category will lead to a different perspective. For a finite dimensional and hereditary $\mathbb{K}$-algebra $A$, we can use work of Ingalls and Thomas (\cite{IT}), who classified the finitely generated wide subcategories of $A\mbox{-}mod$ with respect to support tilting modules. In this context, we obtain the following result.
\\

\noindent\textbf{Theorem A} (Theorem \ref{hered tilt}, Proposition \ref{prop split}) \textit{ Let $A$ be a finite dimensional and hereditary $\mathbb{K}$-algebra. There are bijections (related by restriction) between:
\begin{enumerate}
\item the set of equivalence classes of finitely generated support tilting $A$-modules and the set of epiclasses of finite dimensional universal localisations of $A$;
\item the set of equivalence classes of finitely generated tilting $A$-modules and the set of epiclasses of finite dimensional and monomorphic universal localisations of $A$;
\item the set of equivalence classes of finitely generated support tilting $A/AeA$-modules for an idempotent $e$ in $A$ and the set of epiclasses of finite dimensional universal localisations of $A$ with $A_{\Sigma}\otimes_A Ae=0$.
\end{enumerate}
Moreover, the universal localisation associated to a tilting $A$-module $T$ is given by localising at the set of non split-projective indecomposable direct summands of $T$ in its torsion class.}\\

As a consequence, all finitely generated tilting $A$-modules are (up to equivalence) of the form $A_{\Sigma}\oplus A_{\Sigma}/A$ for some finite dimensional and monomorphic universal localisation $A_{\Sigma}$ of $A$ (see Corollary \ref{cortilt}).
In the tame case, this corollary completes the classification of tilting modules started in \cite{AS2}. Note that not every infinitely generated tilting module over a tame hereditary $\mathbb{K}$-algebra arises from universal localisation (\cite{AS2}). 

Leaving the hereditary case, we will concentrate on universal localisations for Nakayama algebras. On the one hand, these algebras are sufficiently well understood and they share particularly nice homological properties. On the other hand, from a representation theoretical point of view Nakayama algebras are far away from hereditary algebras. They allow to approach universal localisations from a different perspective and may help to get a clearer picture in the general setting. Using work of Dichev (\cite{D}) on the wide subcategories of $A\mbox{-}mod$ for a Nakayama algebra $A$, we obtain a complete classification of the universal localisations by considering orthogonal collections of indecomposable $A$-modules.\\

\noindent\textbf{Theorem B} (Theorem \ref{Main Nak}, Corollary \ref{locNak}) \textit{Let $A$ be a Nakayama algebra. There is a bijection between the universal localisations of $A$ (up to epiclasses) and the sets $\{X_1,...,X_s\}$ of indecomposable $A$-modules (up to isomorphism) with $End_A(X_i)\cong\mathbb{K}$\, for all $i$ and $Hom_A(X_i,X_j)=0$\, for all $i\not= j$. Moreover, a ring epimorphism $A\rightarrow B$ is a universal localisation if and only if $Tor_1^A(B,B)=0$.}\\

In particular, all homological ring epimorphisms are universal localisations. As an application, we provide a combinatorial classification of the homological ring epimorphisms for self-injective Nakayama algebras (see Theorem \ref{class hom}). This result might be of its own interest to study the structure of the derived category of $A$-modules, as suggested in \cite[\S 7.2, Question 4]{CX3}.

Finally, we ask for an analogue of Theorem A for Nakayama algebras, establishing a link between universal localisations and certain generalised tilting modules. It turns out that the right notion for this purpose is given by $\tau$-tilting modules, recently introduced by Adachi, Iyama and Reiten (\cite{AIR}) in order to complete classical tilting theory from the perspective of mutation (compare Section \ref{tilting}). In a first step, we prove a correspondence between the torsion classes and the wide subcategories of our initial module category (Proposition \ref{torsNak}), which gives rise to a bijection between $\tau$-tilting modules and universal localisations.\\

\noindent\textbf{Theorem C} (Theorem \ref{Naktilt}, Theorem \ref{Nak comp}) \textit{ Let $A$ be a Nakayama algebra. There are bijections (related by restriction) between:
\begin{enumerate}
\item the set of equivalence classes of support $\tau$-tilting $A$-modules and the set of epiclasses of universal localisations of $A$;
\item the set of equivalence classes of $\tau$-tilting $A$-modules and the set of epiclasses of universal localisations of $A$ with $A_{\Sigma}\otimes_A Ae\not= 0$ for all idempotents $e\not= 0$ in $A$;
\item the set of equivalence classes of support $\tau$-tilting $A/AeA$-modules for an idempotent $e$ in $A$ and the set of epiclasses of universal localisations of $A$ with $A_{\Sigma}\otimes_A Ae=0$.
\end{enumerate}
Moreover, the universal localisation associated to a $\tau$-tilting module $T$ is given by localising at the set of non split-projective indecomposable direct summands of $T$ in its torsion class.}\\

As a consequence, we can translate some of the combinatorics for universal localisations (see, for example, Corollary \ref{cor pic}) to the theory of $\tau$-tilting modules for Nakayama algebras. Note that even though Theorem A and Theorem C are very similar in nature, their proofs differ significantly. This is partly due to the fact that already the classification of the universal localisations uses different techniques in both cases. Moreover, it turns out that for a Nakayama algebra $A$ the correspondence between the wide subcategories and the torsion classes in $A\mbox{-}mod$ is more involved when compared to the hereditary case.\\

The paper is organized as follows. In Section \ref{Pre}, we start by fixing some notation and introducing the main concepts needed later on. Section \ref{UNIFDA} discusses universal localisations for finite dimensional $\mathbb{K}$-algebras and contains some partial answers to the question of which (homological) ring epimorphisms are given by universal localisations (see Lemma \ref{local} for local algebras and Corollary \ref{cor hom} for some self-injective algebras). Note that we will be mainly interested in universal localisations $A\rightarrow A_{\Sigma}$, where $A_{\Sigma}$ is again a finite dimensional $\mathbb{K}$-algebra. In Section \ref{TULHA}, we prove Theorem A and discuss some consequences of the established bijections. Section \ref{NAKUNI} is devoted to the proof of Theorem B. As an application, we study homological ring epimorphisms for Nakayama algebras (see Subsection \ref{NAKUNI2}). Finally, in Section \ref{NAKTILT}, we prove Theorem C in two steps. First, we combine work of Adachi, Iyama and Reiten on $\tau$-tilting modules (\cite{AIR}) with some of the results obtained in Section \ref{NAKUNI} to get the wanted bijections (see Corollary \ref{Nak bij}). The rest of the section is devoted to the comparison of a support $\tau$-tilting module and its associated universal localisation.

\section{Preliminaries}\label{Pre}

\subsection{Notation}\label{Pre Not}
Throughout, $A$ will denote a finite dimensional algebra over an algebraically closed field $\mathbb{K}$ and all $A$-modules are assumed to be left modules, unless otherwise stated. The category of all (respectively, all finitely generated) $A$-modules will be denoted by $A\mbox{-}Mod$ (respectively, $A\mbox{-}mod$). By $A\mbox{-}ind$ we denote the set containing one representative of each isomorphism class of finitely generated indecomposable $A$-modules. Subcategories of our initial module category are always assumed to be full and closed under isomorphisms. We say that a subcategory $\Ccal$ of $A\mbox{-}mod$, which is closed under direct summands and (finite) direct sums, has a \textbf{finite generator}, if there is an $A$-module $T$ in $\Ccal$ such that for all $X$ in $\Ccal$ there is some $d\in\mathbb{N}$ and a surjection $T^d\rightarrow X$. By $GenT$ we denote the subcategory of $A\mbox{-}mod$ containing all modules which are generated by an $A$-module $T$ and by $addT$ we denote the subcategory of $A\mbox{-}mod$ consisting of all direct summands of (finite) direct sums of copies of $T$. We call a subcategory $\Ccal$ in $A\mbox{-}mod$
\begin{itemize}
\item \textbf{wide}, if $\Ccal$ is exact abelian and extension-closed;
\item \textbf{f-wide}, if $\Ccal$ is wide and has a finite generator;
\item \textbf{bireflective}, if $\Ccal$ is exact abelian and has a finite generator;
\item \textbf{torsion}, if $\Ccal$ is closed under quotients and extensions;
\item \textbf{f-torsion}, if $\Ccal$ is torsion and has a finite generator.
\end{itemize}
The set of all wide (respectively, f-wide, torsion or f-torsion) subcategories of $A\mbox{-}mod$ is denoted accordingly.
We say that an $A$-module $P$ in $\Ccal$ is \textbf{split-projective}, if all surjective morphisms $X\rightarrow P$ in $\Ccal$ split, and \textbf{Ext-projective}, if $Ext_A^1(P,X)=0$ for all $X$ in $\Ccal$.
For a finitely generated $A$-module $X$ we denote by 
$$\xymatrix{...\ar[r] & P_2^X\ar[r]^{\sigma_1^X} & P_1^X\ar[r]^{\sigma_0^X} & P_0^X\ar[r]^{\pi^X} & X\ar[r] & 0}$$
the minimal projective resolution of $X$ in $A\mbox{-}mod$. We are also interested in certain subcategories which are orthogonal to a subcategory $\Ccal$ of $A\mbox{-}mod$, namely
$$^{\perp}\Ccal:=\{X\in A\mbox{-}mod\mid Hom_A(X,\Ccal)=Ext_A^1(X,\Ccal)=0\};$$
$$\Ccal^{\perp}:=\{X\in A\mbox{-}mod\mid Hom_A(\Ccal,X)=Ext_A^1(\Ccal,X)=0\};$$
$$^*\!\Ccal:=\{X\in A\mbox{-}mod\mid Hom_A(X,\Ccal)=Ext_A^1(X,\Ccal)=Hom_A(\sigma_1^X,\Ccal)=0\}$$
$$=\{X\in A\mbox{-}mod\mid Hom_A(\sigma_0^X,C)\text{ an isomorphism for all } C\in\Ccal\};$$
$$\Ccal^*:=\{X\in A\mbox{-}mod\mid Hom_A(\Ccal,X)=Ext_A^1(\Ccal,X)=Hom_A(\sigma_1^{\Ccal},X)=0\}$$
$$=\{X\in A\mbox{-}mod\mid Hom_A(\sigma_0^C,X)\text{ an isomorphism for all } C\in\Ccal\}.$$
Note that for $^*\!\Ccal$ and $\Ccal^*$ to be well-defined, it is actually necessary to consider \textit{minimal} projective resolutions of the corresponding $A$-modules. Certainly, if $A$ is hereditary, we have that $^{\perp}\Ccal=^*\!\!\!\Ccal$ and $\Ccal^{\perp}=\Ccal^*$.

\subsection{Ring epimorphisms and universal localisations}
Recall that a \textbf{ring epimorphism} is an epimorphism in the category of rings with unit. Two ring epimorphisms $f:A\rightarrow B$ and $g:A\rightarrow C$ are equivalent, if there is a ring isomorphism $h: B\rightarrow C$ such that $g=hf$. We then say that $B$ and $C$ lie in the same epiclass of $A$. We have the following well-known description of a ring epimorphism.

\begin{proposition}[\cite{St}, Proposition XI.1.2]\label{ring epi}
For a ring homomorphism $f:A\rightarrow B$, the following statements are equivalent.
\begin{enumerate}
\item $f$ is a ring epimorphism;
\item The restriction functor $f_*: B\mbox{-}Mod\rightarrow A\mbox{-}Mod$ is fully faithful.
\end{enumerate}
\end{proposition}

Since in our setting $A$ will always denote a finite dimensional $\mathbb{K}$-algebra, a ring epimorphism $f:A\rightarrow B$ also turns $B$ into a $\mathbb{K}$-algebra. We call $f$ \textbf{finite dimensional}, if $B$ is a finite dimensional $\mathbb{K}$-algebra. Note that in this case restriction induces a fully faithful functor
$$f_*:B\mbox{-}mod\rightarrow A\mbox{-}mod.$$
For a ring epimorphism $f:A\rightarrow B$ we denote by $\Xcal_B$ the essential image of the restriction functor in $A\mbox{-}Mod$, respectively $A\mbox{-}mod$, if $f$ is finite dimensional.

\begin{theorem}[\cite{GdP}, Theorem 1.2, \cite{GeLe}, \cite{I}, Theorem 1.6.1]\label{bireflective}
There is a bijection between:
\begin{enumerate}
\item ring epimorphisms $A\ra B$ up to equivalence;
\item bireflective subcategories $\Xcal_B$ of $A\mbox{-}Mod$, i.e., subcategories of $A\mbox{-}Mod$ closed under products, coproducts, kernels and cokernels.
\end{enumerate}
This bijection can be restricted to a bijection between:
\begin{enumerate}
\item finite dimensional ring epimorphisms $A\ra B$ up to equivalence;
\item bireflective subcategories $\Xcal_B$ of $A\mbox{-}mod$.
\end{enumerate}
\end{theorem}

We will be mainly interested in finite dimensional epiclasses of $A$. This allows us to work in $A\mbox{-}mod$ instead of $A\mbox{-}Mod$. Necessary and sufficient conditions for a ring epimorphism to be finite dimensional are given in \cite{GdP} (see Proposition 2.2). In some cases, all epiclasses of $A$ are finite dimensional.

\begin{lemma}[\cite{GdP}, Corollary 2.3]\label{GdP87}
If $A$ is a representation-finite and finite dimensional $\mathbb{K}$-algebra, then all ring epimorphisms $A\ra B$ are finite dimensional. In particular, $B$ is again a representation-finite algebra.
\end{lemma}

Sometimes it turns out to be useful to impose a homological condition on a ring epimorphism. Following \cite{GeLe}, we call a ring epimorphism $f:A\rightarrow B$ \textbf{homological}, if $Tor_i^A(B,B)=0$ for all $i>0$ or, equivalently, $Ext_A^i(M,N)\cong Ext_B^i(M,N)$ for all $B$-modules $M$ and $N$. Note that in this situation restriction on the derived module categories induces a fully faithful functor
$$D(f_*):\Dcal(B\mbox{-}Mod)\rightarrow\Dcal(A\mbox{-}Mod).$$

Examples of ring epimorphisms often arise from localisation techniques.
The following theorem defines and shows the existence of universal localisations.

\begin{theorem}[\cite{Sch}, Theorem 4.1]\label{uni loc}
Let $\Sigma$ be a set of maps between finitely generated projective $A$-modules. There is a ring $A_\Sigma$ and a ring homomorphism $f: A\rightarrow A_\Sigma$ such that
\begin{enumerate}
\item $A_\Sigma \otimes_A \sigma$ is an isomorphism of left $A$-modules for all $\sigma$ in $\Sigma$;
\item every ring homomorphism $g:A\rightarrow B$ such that $B\otimes_A \sigma$ is an isomorphism for all $\sigma$ in $\Sigma\,$ factors in a unique way through $f$, i.e., there is a commutative diagram of the form
\begin{equation}\nonumber
\xymatrix{A\ar[rr]^g\ar[rd]_{f}&&B\\ & A_\Sigma.\ar[ru]_{\exists! \tilde{g}}}
\end{equation}
\end{enumerate}
\end{theorem}

We say that the ring $A_\Sigma$ is the \textbf{universal localisation} of $A$ at $\Sigma$. It is well-known that the homomorphism $f: A\rightarrow A_\Sigma$ is a ring epimorphism, unique up to equivalence, with $Tor_1^A(A_\Sigma,A_\Sigma)=0$ or, equivalently, $Ext_A^1(M,N)\cong Ext_{A_{\Sigma}}^1(M,N)$ for all $A_{\Sigma}$-modules $M$ and $N$ (see \cite{Sch}). However, later on, we will see many examples of universal localisations that do not yield homological ring epimorphisms. We will further the discussion of universal localisations in Section \ref{UNIFDA} in the specific context of finite dimensional algebras.

\subsection{Tilting and \texorpdfstring{$\tau$}{tau}-tilting modules}\label{tilting}
Let us first recall the notion of a (classical) tilting module.

\begin{definition}
We call a finitely generated $A$-module $T$ a \textbf{tilting module}, if
\begin{enumerate}
\item[T1)] $pdT\le 1$;
\item[T2)] $Ext_A^1(T,T)=0$;
\item[T3)] There is a short exact sequence $0\rightarrow A\rightarrow T_1\rightarrow T_2\rightarrow 0$ with $T_1,T_2$ in $addT$.
\end{enumerate}
\end{definition}
Note that for a tilting module $T$ we have that $|T|=|A|$, where $|-|$ counts the number of non-isomorphic indecomposable direct summands of a finitely generated $A$-module. From \cite{AS} we can deduce the following result that connects certain tilting modules to ring epimorphisms.

\begin{proposition}\label{prop tilt}
Let $f:A\rightarrow B$ be a finite dimensional and monomorphic ring epimorphism fulfilling that $Tor_1^A(B,B)=0$. Then the following are equivalent:
\begin{enumerate}
\item The projective dimension of $_AB$ is at most one;
\item $T:=_A\!\!B\oplus_A\!B/A$ is a tilting $A$-module.
\end{enumerate} 
\end{proposition}
We say that a tilting module of this form \textbf{arises from a ring epimorphism}. Note that, by \cite{MV} (also compare Theorem \ref{thm old} in this text), every (finitely generated) tilting module arising from a ring epimorphism actually arises from a universal localisation.

We call an $A$-module $T$ \textbf{support tilting}, if $T$ is a tilting module over the $\mathbb{K}$-algebra $A/AeA$ for some idempotent $e$ in $A$. Clearly, all tilting modules are support tilting. The set of isomorphism classes of basic tilting (respectively, support tilting) $A$-modules will be denoted by $tilt(A)$ (respectively, $s\mbox{-}tilt(A)$). Note that there is a natural way of associating a torsion class to a support tilting module $T$ by considering $GenT$. We say that two support tilting $A$-modules $T$ and $T'$ are equivalent, if $GenT=GenT'$. If $A$ is a hereditary $\mathbb{K}$-algebra, we get the following correspondences:

\begin{theorem}[\cite{IT}, \S 2]\label{Thm TI}
Let $A$ be a finite dimensional hereditary $\mathbb{K}$-algebra. There are bijections between
$$s\mbox{-}tilt(A)\rightarrow f\mbox{-}tors(A)\rightarrow f\mbox{-}wide(A)$$
given by mapping a (basic) support tilting module $T$ to $GenT$ and a finitely generated torsion class $\Tcal$ to 
$$\alpha(\Tcal):=\{X\in\Tcal\mid \forall(g:Y\rightarrow X)\in\Tcal, ker(g)\in\Tcal\}.$$
The inverse is given by assigning to a finitely generated wide subcategory $\Ccal$ the torsion class $Gen\Ccal$ and to a finitely generated torsion class $\Tcal$ the (basic) support tilting module $T$, given by the sum of the indecomposable Ext-projectives in $\Tcal$. Furthermore, the split-projective modules in the torsion class coincide with the projective modules in the wide subcategory.
\end{theorem}

For an arbitrary finite dimensional $\mathbb{K}$-algebra $A$ these bijections, in general, will fail. In order to get a similar classification of the finitely generated torsion classes, we use the notion of a $\tau$-tilting module, following \cite{AIR}. We call a finitely generated $A$-module $M$ \textbf{$\tau$-rigid}, if $Hom_A(M,\tau M)=0$, where $\tau$ denotes the usual \textit{Auslander-Reiten translation} in $A\mbox{-}mod$.

\begin{definition}
We call a finitely generated $A$-module $T$ a \textbf{$\tau$-tilting module}, if
\begin{enumerate}
\item[$\tau$1)] $T$ is $\tau$-rigid;
\item[$\tau$2)] $|T|=|A|$.
\end{enumerate}
\end{definition}

It is not hard to check, using the \textit{Auslander-Reiten duality}, that tilting modules are always $\tau$-tilting and, conversely, that faithful $\tau$-tilting modules are already tilting (\cite{AIR}, Proposition 2.2). Indeed, if $A$ is a hereditary algebra, then $\tau$-tilting $A$-modules are tilting $A$-modules.
Similar to the classical setup, we call an $A$-module $T$ \textbf{support $\tau$-tilting}, if $T$ is a $\tau$-tilting module over the $\mathbb{K}$-algebra $A/AeA$ for some idempotent $e$ in $A$. The set of isomorphism classes of basic $\tau$-tilting (respectively, support $\tau$-tilting) $A$-modules will be denoted by $\tau\mbox{-}tilt(A)$ (respectively, $s\tau\mbox{-}tilt(A)$). Note that every support $\tau$-tilting module $T$ gives rise to a torsion class $GenT$. We say that two support $\tau$-tilting $A$-modules $T$ and $T'$ are equivalent, if $GenT=GenT'$. We get the following correspondence between support $\tau$-tilting modules and finitely generated torsion classes.

\begin{theorem}[\cite{AIR}, Theorem 2.7]\label{AIR1}
Let $A$ be a finite dimensional $\mathbb{K}$-algebra. There is a bijection between
$$ s\tau\mbox{-}tilt(A)\longrightarrow f\mbox{-}tors(A)$$
given by mapping a (basic) support $\tau$-tilting module $T$ to the torsion class $GenT$. Conversely, we assign to a finitely generated torsion class $\Tcal$ the sum of the indecomposable Ext-projectives in $\Tcal$.
\end{theorem}

Finally, the introduction of $\tau$-tilting modules in \cite{AIR} was also motivated by the idea of carrying out tilting theory from the perspective of \textbf{mutation}. Indeed, in \cite{AIR} (see Theorem 2.18) it was shown that every basic almost complete support $\tau$-tilting module is a direct summand of precisely two basic support $\tau$-tilting modules. This completion defines mutation between two support $\tau$-tilting modules and gives rise to a partial order on $s\tau\mbox{-}tilt(A)$. The partial order can be understood by comparing the associated torsion classes (see \cite{AIR}, Section 2.4). More precisely, for two support $\tau$-tilting modules $T_1$ and $T_2$ we have that $T_1\le T_2$, if $GenT_1\subseteq Gen T_2$.

\section{Universal localisations for finite dimensional algebras}\label{UNIFDA}
In what follows, we will discuss some properties of universal localisations for finite dimensional $\mathbb{K}$-algebras.
It turns out that we can define universal localisations with respect to a set of finitely generated $A$-modules. Take $\Ucal\subseteq A\mbox{-}mod$ and denote by $A_{\Ucal}$ the universal localisation of $A$ at the set $\{\sigma_0^X\mid X\in\Ucal\}$. Note that $A_{\Ucal}$ is well-defined, since the minimal projective resolutions are essentially unique. Conversely, if we start with a universal localisation $A_{\Sigma}$ of $A$, we define $\Ucal$ to be the set of cokernels of maps in $\Sigma$ plus, additionally, the set of projective $A$-modules which are sent to zero by some map in $\Sigma$. It follows that $A_{\Sigma}$ and $A_{\Ucal}$ lie in the same epiclass of $A$, since an arbitrary map between finitely generated projective $A$-modules $f:P\rightarrow Q$ only differs from the minimal projective presentation of its cokernel by a trivial extension. Indeed, there are finitely generated projective $A$-modules $P'$ and $Q'$ fitting into the following commutative diagram
$$\xymatrix{P\ar[r]^f\ar[d]^\cong & Q\ar[r]\ar[d]^\cong & cok(f)=:M\ar[r]\ar[d]^{id_M} & 0\\ P_1^M\oplus Q'\oplus P'\ar[r]^{f'} & P_0^M\oplus Q'\ar[r] & M\ar[r] & 0}$$
where the map $f'$ is given by the matrix 
$$\left( \begin{array}{ccc}
\sigma_0^M & 0 & 0 \\
0 & id_{Q'} & 0 \end{array} \right).$$
Therefore, universal localisations of $A$ can be defined with respect to a set of finitely generated $A$-modules. Throughout, we will not distinguish explicitly between localising with respect to a set of maps or a set of modules. However, the meaning of the given set $\Sigma$ will become clear in the specific context.
We call a universal localisation $A_{\Sigma}$ of $A$ 
\begin{itemize}
\item \textbf{pure}, if $A_{\Sigma}\otimes_A Ae\not= 0$ for all idempotents $e\not= 0$ in $A$;
\item \textbf{$e$-annihilating}, if $A_{\Sigma}\otimes_A Ae=0$ for an idempotent $e$ in $A$.
\end{itemize}

The set of all (respectively, all pure, $e$-annihilating or finite dimensional) universal localisations of $A$ (up to epiclasses) will be denoted by $uniloc(A)$ (respectively, $uniloc^p(A)$, $uniloc_e(A)$ or $fd\mbox{-}uniloc(A)$). Note that all these sets are partially ordered by inclusion with respect to the essential image of the restriction functor $\Xcal_{A_{\Sigma}}$.
Some of the finite dimensional universal localisations of $A$ are easy to compute. For example, it is not hard to check that the universal localisation at the projective $A$-module $Ae$ for some idempotent $e$ in $A$ is given by the quotient ring $A/AeA$. In fact, all surjective universal localisations of $A$ are of this form.

\begin{lemma}\label{surjective}
Let $A$ be a finite dimensional $\mathbb{K}$-algebra and let $f:A\rightarrow B$ be a surjective ring epimorphism with $Tor_1^A(B,B)=0$. Then there is an idempotent $e$ in $A$ such that $B$ lies in the same epiclass of $A$ as $A/AeA$.
\end{lemma}
\begin{proof}
Certainly, $ker(f)$ is a two-sided ideal in $A$ and $A/ker(f)$ and $B$ lie in the same epiclass of $A$. Since $A$ is a finite dimensional $\mathbb{K}$-algebra, it suffices to show that $ker(f)$ is an idempotent ideal, which follows from $$0=Tor_1^A(B,B)\cong Tor_1^A(A/ker(f),A/ker(f))\cong ker(f)/ker(f)^2.$$\vspace*{-0.3cm}
\end{proof}

Next, we want to use certain pairs of orthogonal subcategories in $A\mbox{-}mod$ (defined in Section \ref{Pre Not}) to study finite dimensional universal localisations of $A$. Note that some of the following observations could also be stated for arbitrary universal localisations of $A$ by considering suitable subcategories of $A\mbox{-}Mod$. Since, later on, we are mainly interested in finitely generated $A$-modules, we leave this possible generalisation to the reader. Thus, let $A_{\Sigma}$ be a finite dimensional universal localisation of $A$. By \cite{CX1} (see Proposition 3.3), we know that $\Xcal_{A_{\Sigma}}$ is given by $\{X\in A\mbox{-}mod\mid Hom_A(\sigma,X)\text{ an isomorphism for all } \sigma\in\Sigma\}$. It can also be described by $\Sigma^*$, if we understand $\Sigma$ as a suitable set of finitely generated $A$-modules. Since $\Xcal_{A_{\Sigma}}$ is closed under extensions in $A\mbox{-}mod$, by Theorem \ref{bireflective}, we get an injective map
$$\omega:fd\mbox{-}uniloc(A)\longrightarrow f\mbox{-}wide(A)$$
by mapping $A_{\Sigma}$ to $\Xcal_{A_{\Sigma}}=\Sigma^*$. Now we can ask the following questions:

\begin{question}\mbox{}\\ \vspace*{-0.4cm}\label{question}

1. How can we describe the image of $\omega$ in $f\mbox{-}wide(A)$?

2. For which choices of $A$ is the map $\omega$ bijective?
\end{question}

In other words, we ask for those finite dimensional ring epimorphisms $f:A\rightarrow B$ with $Tor_1^A(B,B)=0$ that can be realised as universal localisations of $A$. Note that a very first answer is given by Lemma \ref{surjective}. The following proposition determines a candidate for the (partial) inverse of $\omega$.

\begin{proposition}\label{prop uniloc}
Let $f:A\rightarrow B$ be a finite dimensional ring epimorphism. The following holds.
\begin{enumerate}
\item $^*\!\!\Xcal_B=\{X\in A\mbox{-}mod\mid B\otimes_A\sigma_0^X\text{ an isomorphism}\}$, i.e., $^*\!\!\Xcal_B$ describes those finitely generated $A$-modules whose minimal projective presentation becomes invertible under the action of $B\otimes_A-$.
\item $^*\!\!\Xcal_B$ is closed under finite direct sums, direct summands, extensions and cokernels of injective maps whose cokernel is of projective dimension less or equal to one.
\item $(^*\!\!\Xcal_B)^*$ is a wide subcategory of $A\mbox{-}mod$ with $\Xcal_B\subseteq (^*\!\!\Xcal_B)^*$. Moreover, if $f$ is a universal localisation, then we get $\Xcal_B=(^*\!\!\Xcal_B)^*$.
\end{enumerate}
\end{proposition}

\begin{proof}
ad(1): Since the tensor-functor $B\otimes_A-$ is left adjoint to the restriction functor $f_*$ and $f_*$ induces a full embedding of the associated module categories, we have that $Hom_A(\sigma_0^X,Y)$ is an isomorphism for all $Y$ in $\Xcal_B$ if and only if $Hom_A(B\otimes_A \sigma_0^X,Y)$ is an isomorphism for all $Y$ in $\Xcal_B$. Consequently, if $B\otimes_A \sigma_0^X$ is an isomorphism, then $Hom_A(B\otimes_A \sigma_0^X,Y)$ and, therefore, $Hom_A(\sigma_0^X,Y)$ is an isomorphism for all $Y$ in $\Xcal_B$. 

Conversely, let us assume that $Hom_A(\sigma_0^X,Y)$ and, thus, $Hom_A(B\otimes_A \sigma_0^X,Y)$ is an isomorphism for all $Y$ in $\Xcal_B$. It follows that $Hom_A(B\otimes_A X,Y)=0$ for all $Y$ in $\Xcal_B$ and, therefore, we get $B\otimes_A X=0$. Consequently, the $A$-module homomorphism
$$B\otimes_A \sigma_0^X:B\otimes_A P_1^X\rightarrow B\otimes_A P_0^X$$
is surjective. Indeed, it is split surjective, since $B\otimes_A P_1^X$ and $B\otimes_A P_0^X$ are projective $B$-modules. By assumption, we know that $Hom_A(B\otimes_A \sigma_0^X,B\otimes_A P_1^X)$ is an isomorphism such that the identity map on $B\otimes_A P_1^X$ must factor through $B\otimes_A \sigma_0^X$, turning $B\otimes_A \sigma_0^X$ into an isomorphism of $A$-modules. This finishes (1).

ad(2): Since the minimal projective resolution of a direct sum of finitely generated $A$-modules is given by the direct sum of the minimal projective resolutions of their direct summands, $^*\!\!\Xcal_B$ is closed under (finite) direct sums and summands. On the other hand, by the Horseshoe Lemma, we know that for a short exact sequence of finitely generated $A$-modules
$$\xymatrix{0\ar[r] & X\ar[r] & Y\ar[r] & Z\ar[r] & 0}$$
with $X$ and $Z$ in $^*\!\!\Xcal_B$, by taking the direct sum of the minimal projective presentations of $X$ and $Z$, we get a (not necessarily minimal) projective presentation of $Y$ that becomes invertible under the action of $B\otimes_A-$. Consequently, $Y$ belongs to $^*\!\!\Xcal_B$. Finally, if we assume that in the above sequence $X$ and $Y$ belong to $^*\!\!\Xcal_B$, by applying the contravariant functor $Hom_A(-,V)$ for $V$ in $\Xcal_B$, we get that $Hom_A(Z,\Xcal_B)=Ext_A^1(Z,\Xcal_B)=0$. If we further assume that $Z$ is of projective dimension less or equal to one, we can conclude that $Z$ lies in $^*\!\!\Xcal_B$.

ad(3): We consider the universal localisation of $A$ at $^*\!\!\Xcal_B$. Then $(^*\!\!\Xcal_B)^*$ describes the finitely generated $A$-modules over this localisation. Therefore, $(^*\!\!\Xcal_B)^*$ is wide. The inclusion follows from a straightforward verification. Moreover, if $f$ is a finite dimensional universal localisation, we get that $(^*\!\!\Xcal_B)^*=\Xcal_B$ by (1).
\end{proof}

Let us add some remarks to this proposition. For a finite dimensional universal localisation $A_{\Sigma}$ of $A$ we call the modules in $^*\!\!\Xcal_{A_{\Sigma}}$, according to \cite{Sch3}, \textbf{$A_{\Sigma}$-trivial}. Clearly, when seen as a set of modules, $\Sigma$ is contained in $^*\!\!\Xcal_{A_{\Sigma}}$ and the localisation $A_{\Sigma}$ lies in the same epiclass of $A$ as $A_{^*\!\!\Xcal_{A_{\Sigma}}}$. Consequently, a finite dimensional universal localisation of $A$ is uniquely determined by its $A_{\Sigma}$-trivial modules. The partial order on $fd\mbox{-}uniloc(A)$, given by inclusion of the associated module categories, can be reformulated using these modules. More precisely, for $A_{\Sigma_1}$ and $A_{\Sigma_2}$ in $fd\mbox{-}uniloc(A)$ we have $A_{\Sigma_1}\le A_{\Sigma_2}$ if and only if $^*\!\!\Xcal_{A_{\Sigma_1}}\!\!\supseteq ^*\!\!\!\!\Xcal_{A_{\Sigma_2}}$. Besides, since $^*\!\!\Xcal_{A_{\Sigma}}$ is closed under direct sums and summands, it is enough to focus on the indecomposable $A_\Sigma$-trivial modules. The further closure properties of $^*\!\!\Xcal_{A_{\Sigma}}$ can be used to find a minimal subset among these indecomposable modules that still determines the localisation. But, in general, such a set will not be unique.

Concerning Proposition \ref{prop uniloc}(3), one may consider $A_{^*\!\!\Xcal_B}$ as the best approximation of $B$ by a universal localisation of $A$, even though, a priori, it is not clear that $A_{^*\!\!\Xcal_B}$ is again finite dimensional.
In case it is finite dimensional (for example, if $A$ is a representation finite algebra), then $B$ is the universal localisation of $A$ at $^*\!\!\Xcal_B$ if and only if $\Xcal_B=(^*\!\!\Xcal_B)^*$. In many situations, this provides an explicit condition to decide whether a certain ring epimorphism is a universal localisation. Next, we will collect some answers to Question \ref{question}. The following statement can be deduced from \cite{KS} (see Theorem 6.1) using the language of Proposition \ref{prop uniloc}.

\begin{proposition}\label{cor hered}
Let $A$ be a finite dimensional and hereditary $\mathbb{K}$-algebra. Then we have a bijection
$$\omega:fd\mbox{-}uniloc(A)\longrightarrow f\mbox{-}wide(A)$$
by mapping $A_{\Sigma}$ to $\Xcal_{A_{\Sigma}}=\Sigma^*=\Sigma^{\perp}$. The inverse is given by mapping $\Ccal$ in $f\mbox{-}wide(A)$ to $A_{^*\!\!\Ccal}=A_{^{\perp}\!\!\Ccal}$.
\end{proposition}

In particular, $\omega$ is a bijection for every semisimple finite dimensional $\mathbb{K}$-algebra $A$. In this case, all universal localisations (up to epiclasses) are of the form $A/AeA$ for $e$ an idempotent in $A$ (see Lemma \ref{surjective}).

\begin{lemma}\label{local}
Let $A$ be a finite dimensional and local $\mathbb{K}$-algebra. Then the only finite dimensional ring epimorphisms $A\rightarrow B$ (up to epiclasses) with $Tor_1^A(B,B)=0$ are the identity map on $A$ and the zero map. In particular, for this choice of $A$ the map $\omega$ induces a (trivial) bijection.
\end{lemma}

\begin{proof}
Take a non-zero finite dimensional ring epimorphism $A\rightarrow B$ with $Tor_1^A(B,B)=0$ and let $X$ be an indecomposable $A$-module in $\Xcal_B$. Since $A$ is local, $X$ is either simple or it admits, via a top-to-socle factorisation, a non-trivial endomorphism with kernel $X'$ that again lies in $\Xcal_B$. Since, in the second case, the length of the $A$-module $X'$ is smaller than the length of $X$, by induction, we conclude that the unique simple $A$-module $S$ belongs to $\Xcal_B$. Thus, using that $\Xcal_B$ is closed under extensions in $A\mbox{-}mod$, it actually contains all finitely generated $A$-modules and the ring epimorphism $A\rightarrow B$ is equivalent to the identity map on $A$.
\end{proof}

In Section \ref{NAKUNI}, we will obtain a further classification result for Nakayama algebras (see Corollary \ref{locNak}). Some partial answer to Question \ref{question} can also be given by a result in \cite{MV} (see Theorem 3.3), here stated for finite dimensional algebras.

\begin{theorem}\label{thm old}
Let $A$ be a finite dimensional $\mathbb{K}$-algebra and $f:A\rightarrow B$ be a finite dimensional and homological ring epimorphism such that the projective dimension of $_AB$ is at most 1. Then $f$ is a universal localisation. In particular,
$\Xcal_B$ belongs to the image of $\omega$ and fulfils the condition $\Xcal_B=(^*\!\!\Xcal_B)^*$.
\end{theorem}

There is an immediate corollary for some self-injective $\mathbb{K}$-algebras, motivated by \cite{LY}, which relates to an open conjecture by Tachikawa (\cite{T}, Section 8). He conjectured that for a finite dimensional and self-injective $\mathbb{K}$-algebra $A$ and a finitely generated $A$-module $M$, the condition $Ext_A^i(M,M)=0$ for $i>0$ already implies that $M$ is projective. Recall that a finite dimensional $\mathbb{K}$-algebra $A$ is called \textbf{self-injective}, if the free $A$-module of rank one is also injective. The conjecture was proven for several classes of algebras, e.g.,
\begin{itemize}
\item for group algebras of finite groups (see \cite{S}, Chapter 3);
\item for self-injective algebras of finite representation type (see \cite{S}, Chapter 3);
\item for symmetric algebras with radical cube zero (see \cite{H}, Theorem 3.1);
\item for local and self-injective algebras with radical cube zero (see \cite{H}, Theorem 3.4).
\end{itemize}

\begin{corollary}\label{cor hom}
Let $A$ be a finite dimensional and self-injective $\mathbb{K}$-algebra fulfilling the Tachikawa conjecture.
Then all finite dimensional and homological ring epimorphisms $f:A\rightarrow B$ are universal localisations, turning $B$ into a projective $A$-module. Moreover, the $\mathbb{K}$-algebra $B$ is again self-injective.
\end{corollary}

\begin{proof}
Let $f:A\rightarrow B$ be a finite dimensional and homological ring epimorphism. Since $f$ is homological, we know that $Ext_A^n(B,B)\cong Ext_B^n(B,B)=0$ for all $n>0$. Since, by assumption, the $A$-module $_AB$ is finitely generated and the Tachikawa conjecture holds for $A$, $_AB$ must be projective. Consequently, by Theorem \ref{thm old}, $f$ is a universal localisation. Moreover, since the $\mathbb{K}$-algebra $A$ is self-injective, $_AB$ is also an injective $A$-module. Using that $\Xcal_B$ is a full subcategory of $A\mbox{-}mod$, it follows that $_BB$ is an injective $B$-module and, thus, the $\mathbb{K}$-algebra $B$ is self-injective.
\end{proof}

\begin{remark}
Certain group algebras allow a classification of the finite dimensional universal localisations along these lines. For example, let $A$ be the group algebra over $\mathbb{K}$ of a finite $p$-group for a prime $p$. Then a finite dimensional ring epimorphism is a universal localisation if and only if it is homological. Moreover, mapping a finite dimensional universal localisation $A_{\Sigma}$ to $\Xcal_{A_{\Sigma}}$\, yields a bijection
$$\omega:fd\mbox{-}uniloc(A)\longrightarrow f\mbox{-}wide(A).$$ 
In fact, if the characteristic of $\mathbb{K}$ equals the prime $p$, then $A$ is local and we are in the case of Lemma \ref{local}. Otherwise, by Maschke's theorem, the algebra $A$ is semisimple and the claim follows from Proposition \ref{cor hered}.
\end{remark}

For the sake of completeness, we finish the section with two examples of universal localisations of a finite dimensional algebra which are infinite dimensional over the ground field. Note that this phenomena occurs rather frequently, keeping in mind \cite{NRS}. There it was shown that (up to Morita equivalence) every finitely presented algebra appears as the universal localisation of a finite dimensional algebra.

\begin{example}[\cite{NRS}, Section 1]
Let $B$ be the first Weyl algebra, i.e., $B$ is given as the quotient of $\mathbb{K}\!\!<x,y>$ by the two-sided ideal generated by $xy-yx-1$. In particular, $B$ is infinite dimensional over $\mathbb{K}$. Now consider the bound path algebra $A$ over $\mathbb{K}$ given by the quiver
$$\xymatrix{1\ar@<1ex>[r]^{\alpha_1}\ar@<-1ex>[r]_{\alpha_2} & 2\ar@<1ex>[r]^{\beta_1}\ar@<-1ex>[r]_{\beta_2} & 3\ar@<1ex>[r]^{\gamma_1}\ar@<-1ex>[r]_{\gamma_2} & 4}$$ 
and the two-sided ideal generated by $\gamma_2\beta_1\alpha_1-\gamma_1\beta_1\alpha_2$ and $\gamma_2\beta_2\alpha_1-\gamma_1\beta_2\alpha_2-\gamma_1\beta_1\alpha_1$. Then the universal localisation of $A$ one obtains by inverting the arrows $\alpha_1,\beta_1$ and $\gamma_1$ is given by the matrix algebra $M_4(B)$. Note that all non-trivial modules over the localisation are infinitely generated over $A$. Consequently, the example tells us that to check if a universal localisation of a finite dimensional algebra $A$ is finite dimensional, it is not sufficient to see that the finitely generated $A$-modules over the localisation admit a finite generator.
\end{example}

\begin{remark}
If $A$ is a finite dimensional and hereditary $\mathbb{K}$-algebra, it follows from \cite{KS} (Proposition 4.2) that a universal localisation $A_\Sigma$ of $A$ is finite dimensional if and only if there is a finitely generated $A$-module $X$ with $Ext_A^1(X,X)=0$ such that $A_{\Sigma}$ and $A_{\{X\}}$ lie in the same epiclass of $A$.
\end{remark}

In general, such an $A$-module $X$ will not exist for a given universal localisation.

\begin{example}
Consider the Kronecker algebra $A=\begin{pmatrix} \mathbb{K} & 0 \\ \mathbb{K}^2 & \mathbb{K} \end{pmatrix}$ and a quasi-simple regular $A$-module $S$. In particular, we have $Ext_A^1(S,S)\not= 0$. It is well-known (see, for example, \cite{Sch2}) that the universal localisation of $A$ at $\{S\}$ is given by the matrix algebra $M_2(\mathbb{K}[x])$, which is clearly infinite dimensional over $\mathbb{K}$. Note that the $A$-module structure of $M_2(\mathbb{K}[x])$, induced by the ring epimorphism, depends on the choice of $S$.
\end{example}

\section{Tilting modules and universal localisations for hereditary algebras}\label{TULHA}
We begin this section with a small lemma, stated in \cite{Sch4} without a proof.
\begin{lemma}\label{lem inj}
Let $A$ be a finite dimensional and hereditary $\mathbb{K}$-algebra. Then a universal localisation $A\rightarrow A_{\Sigma}$ is monomorphic if and only if it is pure.
\end{lemma}

\begin{proof}
First, we observe that as a map of $A$-modules we can write the ring homomorphism $f:A\rightarrow A_{\Sigma}$ in the following form:
$$f:A\rightarrow A_{\Sigma}\otimes_A A$$
$$a\mapsto f(a)\otimes 1_A=1_{A_{\Sigma}}\otimes a$$
Now assume that $f$ is monomorphic and suppose there is some idempotent $e\not= 0$ in $A$ with $A_{\Sigma}\otimes_A Ae=0$. It follows that $f(Ae)=1_{A_{\Sigma}}\otimes Ae=0$ and, therefore, $Ae\subseteq ker(f)$, a contradiction. Conversely, assume that the localisation $A_\Sigma$ is pure and suppose that $ker(f)\not= 0$. Take some $x\not= 0$ in $ker(f)$ and consider the left ideal $I$ of $A$ generated by $x$. Clearly, $I\subseteq ker(f)$. Since $A$ is hereditary, $I$ is a projective left $A$-module of the form $Ae$ for some idempotent $e\not= 0$ in $A$. Now it follows that
$0=f(Ae)=1_{A_{\Sigma}}\otimes Ae$ and, thus, we get $A_{\Sigma}\otimes_A Ae=0$, again yielding a contradiction.
\end{proof}

Note that monomorphic universal localisations $A\rightarrow A_{\Sigma}$ are always pure. But the converse will fail in general (compare Example \ref{Ex Nak} and Example \ref{ex not induced}).
In the hereditary case, the following theorem establishes a bijection between support tilting $A$-modules and finite dimensional universal localisations of $A$.

\begin{theorem}\label{hered tilt}
Let $A$ be a finite dimensional and hereditary $\mathbb{K}$-algebra.
\begin{enumerate}
\item There is a bijection
$$\Psi_A:s\mbox{-}tilt(A)\longrightarrow fd\mbox{-}uniloc(A)$$
by mapping a support tilting $A$-module $T$ to $A_{\Sigma_T}\!:=A_{^{\perp}(\alpha(GenT))}$. The inverse is given by mapping a universal localisation $A_{\Sigma}$ to $T_{\Sigma}$, the sum of the indecomposable Ext-projectives in $Gen(\Sigma^{\perp})$.
\item $\Psi_A$ restricts to a bijection between
$$tilt(A)\longrightarrow fd\mbox{-}uniloc^p(A).$$
Moreover, regarding the inverse, $T_{\Sigma}$ is equivalent to $A_{\Sigma}\oplus A_{\Sigma}/A$.
\item $\Psi_A$ restricts to a bijection between
$$s\mbox{-}tilt(A/AeA)\longrightarrow fd\mbox{-}uniloc_e(A)$$
for an idempotent $e$ in $A$. In particular, if $T$ is equivalent to $_A\!(A/AeA)$, it is mapped to $A_{\Sigma_T}=A/AeA$.
\end{enumerate}
\end{theorem}

\begin{proof}
ad(1): Follows from Theorem \ref{Thm TI} and Proposition \ref{cor hered}.

ad(2): First, take a basic tilting $A$-module $T$ and let $P$ be an indecomposable projective $A$-module. We want to show that $Hom_A(P,\alpha(GenT))\not= 0$. Since $T$ is tilting, we have a short exact sequence of the form
$$\xymatrix{0\ar[r] & P\ar[r]^{f'} & T_0\ar[r] & T_1\ar[r] & 0}$$
with $T_0$ and $T_1$ in $addT$. Now suppose that $T_0\notin\alpha(GenT)$. Since, by \cite{IT} (see Proposition 2.15), we know that $\alpha(GenT)$ is given by
$$\{X\in GenT\mid \forall(g:Y\twoheadrightarrow X)\in GenT, Y \text{ split\mbox{-}projective}: ker(g)\in GenT\},$$
there is a split-projective module $Z$ in $GenT$ (in fact, $Z$ lies in $addT$) and a surjection $g:Z\twoheadrightarrow T_0$ such that $ker(g)\notin GenT$. Since $P$ is projective, we can lift the map $f'$ to get an injective map $h:P\rightarrow Z$ with $f'=g\circ h$. But the split-projective modules in $GenT$ must also belong to $\alpha(GenT)$ (see Theorem \ref{Thm TI}) and we get that $Hom_A(P,\alpha(GenT))\not= 0$. Therefore, $P$ does not lie in $^{\perp}(\alpha(GenT))= ^{\perp}\!\!\Xcal_{A_{\Sigma_T}}$. It follows that $A_{\Sigma_T}$ is pure.

Conversely, let $A_{\Sigma}$ be a pure and finite dimensional universal localisation of $A$. By Lemma \ref{lem inj}, the morphism $f:A\rightarrow A_{\Sigma}$ is monomorphic and we get the following short exact sequence of $A$-modules
$$\xymatrix{0\ar[r] & A\ar[r]^f & A_{\Sigma}\ar[r] & A_{\Sigma}/A \ar[r] & 0}.$$
We already know, by Proposition \ref{prop tilt}, that $T_{\Sigma}':=A_{\Sigma}\oplus A_{\Sigma}/A$ is a tilting $A$-module. Therefore, it suffices to show that $GenT_{\Sigma}=GenT_{\Sigma}'$. This follows from Theorem \ref{Thm TI} and the construction of $\Psi_A$ in (1), since
$$GenT_{\Sigma}\overset{Thm. \ref{Thm TI}}{=}Gen(\alpha(GenT_{\Sigma}))\overset{(1)}{=}Gen\Xcal_{A_{\Sigma}}=GenA_{\Sigma}=GenT_{\Sigma}'.$$

ad(3): For a given idempotent $e$ in $A$, a basic support tilting $A$-module $T$ belongs to $s\mbox{-}tilt(A/AeA)$ if and only if $T$ carries the natural structure of an $A/AeA$-module (i.e., $T\in\Xcal_{A/AeA}$) or, equivalently, we have $Hom_A(Ae,T)=0$. Similar to the first implication in (2), one can show that $A_{\Sigma_T}\otimes_A Ae=0$ implies $Hom_A(Ae,T)=0$. In other words, if the localisation $A_{\Sigma_T}$ is $e$-annihilating, then $T$ belongs to $s\mbox{-}tilt(A/AeA)$.

Conversely, if $Hom_A(Ae,T)=0$, we get that $Hom_A(Ae,GenT)=0$, since $Ae$ is projective. In particular, $Hom_A(Ae,\alpha(GenT))$ must be zero. It follows that $Ae$ lies in $^{\perp}(\alpha(GenT))= ^{\perp}\!\!\Xcal_{A_{\Sigma_T}}$ such that $A_{\Sigma_T}\otimes_A Ae=0$. Altogether, $T$ belongs to $s\mbox{-}tilt(A/AeA)$ if and only if $A_{\Sigma_T}$ is $e$-annihilating. 
Finally, if $T$ is equivalent to $_A(A/AeA)$, then $GenT$ is already abelian and we get the following chain of equalities
$$\Xcal_{A_{\Sigma_T}}=\alpha(GenT)=GenT=Gen(A/AeA)=\Xcal_{A/AeA}$$
such that $A_{\Sigma_T}$ and $A/AeA$ lie in the same epiclass of $A$.
\end{proof}

\begin{corollary}\label{cor hered tilt}
Let $A$ be as above. For an idempotent $e$ in $A$ there is a commutative diagram of bijections 
$$\xymatrix{fd\mbox{-}uniloc_e(A)\ar[rr]^{\Phi_e} & & fd\mbox{-}uniloc(A/AeA) \\ & s\mbox{-}tilt(A/AeA)\ar[ul]^{\Psi_A}\ar[ur]_{\Psi_{A/AeA}} &}$$
where $\Phi_e$ maps a universal localisation $A_{\Sigma}$ of $A$ to the universal localisation of $A/AeA$ at the set $$\!^\perp\Xcal_{A_{\Sigma}}\cap\Xcal_{A/AeA}$$ 
of finitely generated $A/AeA$-modules. The inverse is given by mapping a universal localisation $(A/AeA)_{\Sigma'}$ of $A/AeA$ to the universal localisation of $A$ at the set $\Sigma'\cup\{Ae\}$ of finitely generated $A$-modules.
\end{corollary}

\begin{proof}
On the one hand, by Theorem \ref{hered tilt}(3), we can identify the (basic) support tilting $A/AeA$-modules via $\Psi_A$ with the finite dimensional and $e$-annihilating universal localisations of $A$. On the other hand, by applying Theorem \ref{hered tilt}(1) to the finite dimensional and hereditary $\mathbb{K}$-algebra $A/AeA$, the map $\Psi_{A/AeA}$ describes a bijection between the (basic) support tilting $A/AeA$-modules and the finite dimensional universal localisations of $A/AeA$. The map $\Phi_e$ is now defined as the composition $\Psi_{A/AeA}\circ\Psi_A^{-1}$. The precise assignments follow from the construction.
\end{proof}

\begin{remark}\label{rem hered}
The inverse of the map $\Psi_A$ in Theorem \ref{hered tilt}(1) can also be expressed as follows: Take a finite dimensional universal localisation $A_{\Sigma}$ of $A$. Either $A_{\Sigma}$ is already pure and, thus, $T_{\Sigma}$ is equivalent to $A_{\Sigma}\oplus A_{\Sigma}/A$ or it exists an idempotent $e$ in $A$ such that $A_{\Sigma}$ is $e$-annihilating and the universal localisation $\Phi_e(A_{\Sigma})$ of $A/AeA$ is pure. Consequently, by Corollary \ref{cor hered tilt} and Theorem \ref{hered tilt}(2), $T_{\Sigma}$ is equivalent to the support tilting $A$-module $\Phi_e(A_{\Sigma})\oplus \Phi_e(A_{\Sigma})/(A/AeA)$, where $\Phi_e(A_{\Sigma})$, regarded as an $A$-module, is isomorphic to $_AA_{\Sigma}$.
\end{remark}

\begin{corollary}\label{cortilt}
For a finite dimensional and hereditary $\mathbb{K}$-algebra $A$ every tilting $A$-module (up to equivalence) arises from a universal localisation. In particular, for every tilting $A$-module $T$ there is a short exact sequence $0\rightarrow A\rightarrow T_1\rightarrow T_2\rightarrow 0$ with $T_1,T_2$ in $addT$ and $Hom_A(T_2,T_1)=0$.
\end{corollary}

\begin{proof}
Follows from Theorem \ref{hered tilt}(2) and \cite{AS} (see Theorem 3.10). The short exact sequence
$$0\rightarrow A\rightarrow A_{\Sigma_T}\rightarrow A_{\Sigma_T}/A\rightarrow 0$$ fulfils the wanted properties.
\end{proof}

We will further the comparison between a support tilting $A$-module and its associated universal localisation. The following proposition tells us how to read off $A_{\Sigma_T}$ from the $A$-module $T$. 

\begin{proposition}\label{prop split}
Let $A$ be a finite dimensional and hereditary $\mathbb{K}$-algebra and $T$ be a basic support tilting $A$-module which is tilting over the algebra $A/AeA$ for an idempotent $e$ in $A$. Then $A_{\Sigma_T}$ is given by localising at the set of all non split-projective indecomposable direct summands of $T$ in $GenT$ and the $A$-module $Ae$.
\end{proposition}

\begin{proof}
By Corollary \ref{cor hered tilt} and Remark \ref{rem hered}, $T$ is equivalent to the support tilting $A$-module 
$$\Phi_e(A_{\Sigma_T})\oplus \Phi_e(A_{\Sigma_T})/(A/AeA),$$
induced by the finite dimensional and pure, thus, monomorphic universal localisation $\Phi_e(A_{\Sigma_T})$ of $A/AeA$. Note that if $T$ is a tilting $A$-module, then $e$ is zero and $\Phi_e$ equals the identity on $fd\mbox{-}uniloc(A)$. By \cite{MV} (see Corollary 3.7), $\Phi_e(A_{\Sigma_T})$ is given by localising with respect to the finitely generated $A/AeA$-module 
$$\Phi_e(A_{\Sigma_T})/(A/AeA)$$ 
and, thus, by Corollary \ref{cor hered tilt}, $A_{\Sigma_T}$ is given by localising at the set $\{\Phi_e(A_{\Sigma_T})/(A/AeA),Ae\}$ of $A$-modules. Consequently, it remains to show that $\Phi_e(A_{\Sigma_T})/(A/AeA)$, viewed as an $A$-module, describes precisely the non split-projective indecomposable direct summands of $T$ in $GenT$. Indeed, one can show that an indecomposable $A$-module in $addT$ is not split-projective in $GenT$ if and only if it belongs to $add\{\Phi_e(A_{\Sigma_T})/(A/AeA)\}$. One implication follows from the fact that the split-projective $A$-modules in $GenT$ are precisely given by $add\{A_{\Sigma_T}\}=add\{\Phi_e(A_{\Sigma_T})\}$ (see Theorem \ref{Thm TI}). For the other implication we additionally observe that there are no $A$-homomorphisms from $\Phi_e(A_{\Sigma_T})/(A/AeA)$ to $\Phi_e(A_{\Sigma_T})$ (see Corollary \ref{cortilt}) such that the $A$-module $\Phi_e(A_{\Sigma_T})/(A/AeA)$ cannot have any indecomposable direct summand which is split-projective in $GenT$.
\end{proof}

In particular, if $T$ is a (basic) tilting $A$-module, then $A_{\Sigma_T}$ is given by localising at the set of the non split-projective indecomposable direct summands of $T$ in $GenT$.

\begin{example}
Let $A$ be a finite dimensional basic and hereditary $\mathbb{K}$-algebra with a sink in the underlying quiver. Let $S$ be a simple and projective $A$-module which is not injective. We write $_AA$ as a direct sum $P\oplus S$ of projective $A$-modules. By $\tau$ we denote the usual \textit{Auslander-Reiten translation}. Then the $A$-module
$$T:=\tau^{-1} S\oplus P$$
is tilting, following \cite{APR}. $T$ is usually called an APR-tilting module. Using Proposition \ref{prop split}, we conclude that the associated universal localisation $A_{\Sigma_T}$ of $A$ is given by $A_{\{\tau^{-1}S\}}$. 
\end{example}

In the last part of this section we will discuss how the notion of mutation for support tilting modules or, more precisely, the induced partial order (see Section \ref{tilting}) translates to the set of universal localisations. Again, by $A$ we denote a finite dimensional and hereditary $\mathbb{K}$-algebra. It is not hard to see that the partial order on $s\mbox{-}tilt(A)$, given by inclusion of the associated torsion classes, is finer than the natural partial order on $fd\mbox{-}uniloc(A)$. Indeed, if $A_{\Sigma_1}$ and $A_{\Sigma_2}$ are finite dimensional universal localisations of $A$ with $A_{\Sigma_1}\le A_{\Sigma_2}$, then we have $\Xcal_{A_{\Sigma_1}}\subseteq\Xcal_{A_{\Sigma_2}}$ and, therefore, $Gen(\Xcal_{A_{\Sigma_1}})\subseteq Gen(\Xcal_{A_{\Sigma_2}})$, showing that for the associated support tilting modules $T_{\Sigma_1}$ and $T_{\Sigma_2}$ it follows $T_{\Sigma_1}\le T_{\Sigma_2}$. However, the following easy example illustrates that the converse, in general, does not hold true.

\begin{example}\label{exA2}
Consider the path algebra $A:=\mathbb{K}(1\rightarrow 2)$ and the two support tilting $A$-modules $T_1:=P_1\oplus S_1$ and $T_2:=S_1$, which are clearly mutations of each other. We have $T_2\le T_1$. But the associated universal localisations $A_{\Sigma_{T_1}}$ and $A_{\Sigma_{T_2}}$ are not related. Indeed, $A_{\Sigma_{T_1}}$ is the universal localisation of $A$ at $\{S_1\}$, where $A_{\Sigma_{T_2}}$ is the localisation at $\{P_2\}$. We can also compare the Hasse quivers for the different partial orders on $s\mbox{-}tilt(A)$ and $fd\mbox{-}uniloc(A)$. The first one is given by
$$
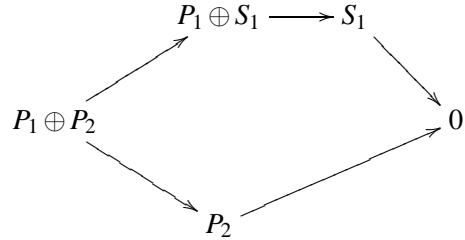
\xymatrix{& P_1\oplus S_1\ar[r] & S_1\ar[dr] & \\ P_1\oplus P_2\ar[ur]\ar[dr] & & & 0 \\ & P_2\ar[urr] & &}$$
\noindent
In the Hasse quiver for the natural partial order on $fd\mbox{-}uniloc(A)$ the universal localisations of $A$ are indicated by the corresponding indecomposable $A_{\Sigma}$-trivial modules (see Section \ref{UNIFDA}).

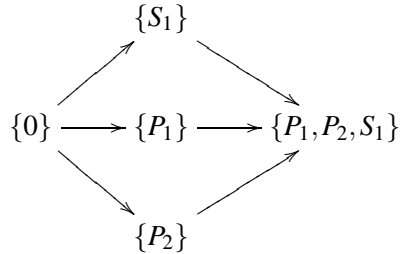
$$\xymatrix{& \{S_1\}\ar[dr] & \\ \{0\}\ar[ur]\ar[r]\ar[dr] & \{P_1\}\ar[r] & \{P_1,P_2,S_1\}\\ & \{P_2\}\ar[ur] &}$$
\end{example}

\section{Nakayama algebras and universal localisations}\label{NAKUNI}
In this section we classify the universal localisations of a Nakayama algebra $A$ by certain subcategories of $A\mbox{-}mod$. More precisely, we show that the map $\omega$ discussed in Section \ref{UNIFDA} is bijective (see Question \ref{question}). Note that for Nakayama algebras all universal localisations are finite dimensional and all (relevant) subcategories of $A\mbox{-}mod$ have a finite generator, since $A$ is representation finite (see Proposition \ref{Nakayama}). In this section, $l(X)$ denotes the \textbf{Loewy length} of a finitely generated $A$-module $X$. We first recall the definition of a Nakayama algebra.
A finite dimensional $\mathbb{K}$-algebra $A$ is called \textbf{Nakayama} if every indecomposable projective $A$-module and every indecomposable injective $A$-module is uniserial. The following well-known result helps to understand the representation theory of $A$.

\begin{proposition}[\cite{ASS}, Theorem V.3.5]\label{Nakayama}
Let $A$ be a Nakayama algebra and $M$ an indecomposable $A$-module. Then it exists an indecomposable projective $A$-module $P$ and a positive integer $t$ with $1\le t\le l(P)$ such that $M\cong P/rad^tP$. In particular, $A$ is representation finite and every indecomposable $A$-module is uniserial.
\end{proposition}

We want to realise Nakayama algebras as bound path algebras. Consider for $n\in\mathbb{N}$ the quivers
$$\Delta_n:=\xymatrix{1\ar[r] & 2\ar[r] & 3\ar[r] & ... \ar[r] & n}$$
$$\tilde{\Delta}_n:=\xymatrix{1\ar[r] & 2\ar[r] & 3\ar[r] & ... \ar[r] & n.\ar@/^1pc/[llll]}$$
\medskip

It is a well-known fact that a basic and connected $\mathbb{K}$-algebra $A$ is a Nakayama algebra if and only if $A$ is isomorphic to a quotient $\mathbb{K}Q_A/I$, where $Q_A$ is a quiver of the form $\Delta_n$ or $\tilde{\Delta}_n$ and $I$ is an admissible ideal of $\mathbb{K}Q_A$. Moreover, $A$ is a self-injective Nakayama algebra not isomorphic to the field if and only if $Q_A=\tilde{\Delta}_n$ and the admissible ideal $I$ is a power of the arrow ideal of $\mathbb{K}Q_A$ (see \cite{ASS}, Chapter V.3).

Later on, the following Nakayama algebras will play an important role 
$$A_n^h:=\mathbb{K}\Delta_n/R^h\hspace{2pc}\text{and}\hspace{2pc}\tilde{A}_n^h:=\mathbb{K}\tilde{\Delta}_n/R^h,$$
where $h\in\mathbb{N}_{>1}$ and $R$ denotes the arrow ideal of the associated path algebra.
The following lemma will be useful throughout.
\begin{lemma}[\cite{D}, Lemma 2.2.2]\label{Nakext}
Let $A$ be a Nakayama algebra and $X_1,X_2$ indecomposable $A$-modules. If 
$$\xymatrix{0\ar[r] & X_1\ar[r] & Y\ar[r] & X_2\ar[r] & 0}$$ 
is a non-split short exact sequence of $A$-modules, then $Y$ has at most two indecomposable direct summands $Y_1$ and $Y_2$ and the short exact sequence is of the form
$$\xymatrix{& & Y_1\ar@{>>}[dr] & & \\ 0\ar[r] & X_1\ar@{^{(}->}[ur]\ar@{>>}[dr] & & X_2\ar[r] & 0.\\ & & Y_2\ar@{^{(}->}[ur] & &}$$
\end{lemma}

Next, we want to understand universal localisations of Nakayama algebras.

\begin{lemma}\label{lem Nak1}
Let $A$ be a Nakayama algebra and $A_{\Sigma}$ be a universal localisation of $A$.
Then also $A_{\Sigma}$ is a Nakayama algebra.
\end{lemma}

\begin{proof}
First of all, by Proposition \ref{Nakayama} and Lemma \ref{GdP87}, $A_{\Sigma}$ is a finite dimensional and representation finite $\mathbb{K}$-algebra. Now let $X$ be an indecomposable projective or injective $A_{\Sigma}$-module. Since, via restriction, $_AX$ is an indecomposable (not necessarily projective or injective) $A$-module, it is uniserial by Proposition \ref{Nakayama}. Consequently, $X$ is uniserial as an $A_{\Sigma}$-module. 
\end{proof}

\begin{remark}
In general, a universal localisation $A_{\Sigma}$ of a basic and connected Nakayama algebra $A$ is neither basic nor connected.
\end{remark}

\begin{example}\label{Ex Nak}
Consider the Nakayama algebra $A:=A_3^2$ and the universal localisation at $\Sigma:=\{S_2\}$, which one obtains by inverting the arrow $2\rightarrow 3$ in the quiver $\Delta_3$. The $A$-module $_AA_{\Sigma}$ is five-dimensional of the form $P_2^{\oplus 2}\oplus S_1$ and the algebra $A_{\Sigma}$ is Morita-equivalent to $\mathbb{K}\times\mathbb{K}$. In particular, $A_{\Sigma}$ is neither basic nor connected.
\end{example}

In order to classify universal localisations for Nakayama algebras we use some of the methods developed in Section \ref{UNIFDA}. By Proposition \ref{prop uniloc}, we know that a universal localisation $A_{\Sigma}$ of an algebra $A$ is determined by the indecomposable $A$-modules in $^*\!\Xcal_{A_{\Sigma}}$. In the Nakayama case we can be more precise. We will consider a minimal and explicitly given set of indecomposable $A_\Sigma$-trivial modules which determines the localisation. Let $n$ be the number of non-isomorphic simple $A$-modules, $i\in\{1,...,n\}$ and $P_i$ be the corresponding indecomposable projective $A$-module. Then we define $X_i^{\Sigma}$ to be $P_i/rad^{t_i}P_i$ for $t_i\ge 0$ minimal such that $P_i/rad^{t_i}P_i$ lies in $\, ^*\!\!\Xcal_{A_{\Sigma}}$, whenever such $t_i$ exists. By convention, we define $P_i/rad^0P_i$ to be $P_i$. The set of all the indecomposable $X_i^{\Sigma}$ is denoted by $W_{\Sigma}$.

\begin{lemma}\label{lemNak}
Let $A$ be a Nakayama algebra and $A_{\Sigma}$ be a universal localisation of $A$. Then $A_{\Sigma}$ is uniquely determined by the set $W_{\Sigma}$, i.e., $A_{\Sigma}$ and $A_{W_{\Sigma}}$ lie in the same epiclass of $A$.
\end{lemma}

\begin{proof}
We have to show that $^*\!\Xcal_{A_{\Sigma}}$ equals $^*\!\Xcal_{A_{W_\Sigma}}$. However, one of the inclusions is immediate. Thus, let $X$ be an indecomposable $A$-module in $^*\!\Xcal_{A_{\Sigma}}$ and take the minimal projective presentation $P_1^X\rightarrow P_0^X$ of $X$ in $A\mbox{-}mod$. We can assume that $X$ is not projective (otherwise it would already belong to $W_{\Sigma}$). Then either $P_0^X$ belongs to $W_{\Sigma}$ or there is $j\in\{1,...,n\}$ such that $X$ surjects onto $X_{j}^{\Sigma}=P_{j}/rad^{t_j}P_{j}$ for some $t_{j}\ge 1$. In the first case, since $X$ belongs to $^*\!\Xcal_{A_{\Sigma}}$, also $P_1^X$ has to be in $W_{\Sigma}$. Consequently, by definition, $X$ belongs to $^*\!\Xcal_{A_{W_\Sigma}}$. In the second case, we get a short exact sequence of the form
$$\xymatrix{0\ar[r] & ker(\pi)\ar[r] & X\ar[r]^{\pi} & X_{j}^{\Sigma}\ar[r] & 0}$$
and it is easy to check, by comparing minimal projective presentations, that also $ker(\pi)$ lies in $^*\!\Xcal_{A_{\Sigma}}$. By Proposition \ref{prop uniloc}(2), $^*\!\Xcal_{A_{W_\Sigma}}$ is closed under extensions and, therefore, $X$ lies in $^*\!\Xcal_{A_{W_\Sigma}}$ if and only if $ker(\pi)$ belongs to $^*\!\Xcal_{A_{W_\Sigma}}$. But now we can repeat the whole argument with $ker(\pi)$ instead of $X$ and, since the length of $ker(\pi)$ is smaller than the length of $X$, we are done after finitely many steps.
\end{proof}

Note that a non-projective $A$-module $X_i^{\Sigma}$ in $W_{\Sigma}$ represents the "shortest" non-trivial morphism, from an indecomposable projective $A$-module $P_j$ to the module $P_i$, which becomes invertible after tensoring with $A_{\Sigma}$. By "shortest" we mean a minimal number of factorisations through other indecomposable $A$-modules.
The classification of the universal localisations of $A$ will work via the notion of orthogonal collections. We call a set of $A$-modules $\{X_1,...,X_s\}$ an \textbf{orthogonal collection}, if every $X_i$ is indecomposable, $End_A(X_i)\cong\mathbb{K}$ for all $i$ and $Hom_A(X_i,X_j)=0$ for all $i\not= j$. Since $A$ is a Nakayama algebra, we clearly have that $s\le n$ and that $s=n$ if and only if all $X_i$ are simple $A$-modules. The following proposition can be deduced from \cite{D}.

\begin{proposition}[\cite{D}, Proposition 2.2.8, Theorem 2.2.10, \S 2.6]\label{Prop Nak}
Let $A$ be a Nakayama algebra. There is a bijection between the wide subcategories and the isomorphism classes of orthogonal collections in $A\mbox{-}mod$ by mapping a wide subcategory $\Ccal$ to the set of $\Ccal$-simple $A$-modules.
\end{proposition}

Now we are able to state the main result of this section.

\begin{theorem} \label{Main Nak}
Let $A$ be a Nakayama algebra. There is a bijection between the universal localisations of $A$ (up to epiclasses) and the isomorphism classes of orthogonal collections in $A\mbox{-}mod$.
\end{theorem}

\begin{proof}
Let $A_{\Sigma}$ be a universal localisation of $A$. By Lemma \ref{lemNak}, $A_{\Sigma}$ is uniquely determined by the set $W_{\Sigma}$. The general idea of the proof is to deform $W_{\Sigma}$ uniquely into an orthogonal collection of $A$-modules and to show that every orthogonal collection of $A$-modules occurs in this way.
In a first step, we list and prove five important properties of $W_{\Sigma}$:
\begin{enumerate}
\item For all non-projective $X$ in $W_{\Sigma}$, we have $l(X)\le n-1$.
\item Composing minimal projective presentations of modules in $W_{\Sigma}$ can never yield an endomorphism.
\item For all non-projective $X$ in $W_{\Sigma}$, the minimal projective presentation $\sigma_0^X$ of $X$ does not factor through any projective $A$-module in $W_{\Sigma}$.
\item The minimal projective presentations of two different non-projective modules in $W_{\Sigma}$ cannot have the same domain.
\item The minimal projective presentation of a non-projective $A$-module in $W_{\Sigma}$ factors properly through the projective cover $P_0^X$ of a non-projective $A$-module $X$ in $W_{\Sigma}$ if and only if it factors through $P_1^X$.
\end{enumerate}
ad(1): Suppose that $l(X)\ge n$. Consequently, the minimal projective presentation of $X$ 
$$\sigma_0^{X}:P_1^{X}\rightarrow P_0^{X}$$ factors through a non-trivial endomorphism $\alpha$ of $P_1^{X}$ (respectively, through a non-trivial endomorphism of $P_0^{X}$). Since $X$ belongs to $^*\!\Xcal_{A_{\Sigma}}$, the morphism $\sigma_0^{X}$ becomes invertible after tensoring with $A_{\Sigma}$ and, thus, also does $\alpha$. Note that $A_{\Sigma}\otimes_A\alpha$ is not zero by assumption. This leads to a contradiction, since $A$ is a finite dimensional algebra and, therefore, all non-trivial endomorphisms of indecomposable projective $A$-modules must be nilpotent.
ad(2): Can be checked using the arguments in the proof of (1).  
ad(3): Suppose the map $\sigma_0^{X}$ factors through some $P_j$ in $W_{\Sigma}$. Since $P_j$ is getting annihilated by tensoring with $A_{\Sigma}$ while the map $\sigma_0^{X}$ becomes invertible at the same time, it follows that also the projective cover $P_0^X$ of $X$ gets annihilated and, thus, $P_0^X$ must belong to $^*\!\Xcal_{A_{\Sigma}}$. Therefore, by the definition of $W_{\Sigma}$, $X$ is projective, contradicting our assumption.
ad(4): Suppose that the negation of (4) holds for two non-projective $A$-modules $X_1$ and $X_2$ in $W_{\Sigma}$. Hence, without loss of generality, we can assume that the minimal projective presentation of $X_1$
$$\sigma_0^{X_1}:P_1^{X_1}\rightarrow P_0^{X_1}$$ 
factors non-trivially through $P_0^{X_2}$, the projective cover of $X_2$. Since $X_1$ and $X_2$ belong to $^*\!\Xcal_{A_{\Sigma}}$, also the induced map from $P_0^{X_2}$ to $P_0^{X_1}$ becomes invertible under the action of $A_{\Sigma}\otimes_A-$, contradicting the minimality of $X_1$ in the definition of $W_{\Sigma}$. 
ad(5): Let $X_1$ and $X_2$ be two non-projective $A$-modules in $W_{\Sigma}$ such that the minimal projective presentation of $X_1$ factors properly through $P_0^{X_2}$, the projective cover of $X_2$. Now suppose that condition (5) is not fulfilled. Thus, we get the following commutative diagram of $A$-modules
$$\xymatrix{ & P_1^{X_1}\ar[rr]^{\sigma_0^{X_1}}\ar[dr]^{f_2} & & P_0^{X_1}\\ P_1^{X_2}\ar[rr]^{\sigma_0^{X_2}}\ar[ur]^{f_1} & & P_0^{X_2}\ar[ur]_{f_3} &}$$
Since $X_1$ and $X_2$ belong to $^*\!\Xcal_{A_{\Sigma}}$, also the cokernels of the $f_i$ lie in $^*\!\Xcal_{A_{\Sigma}}$, again contradicting the minimality of $X_1$ and $X_2$ in the definition of $W_{\Sigma}$. Since the argument is symmetric, the reverse implication follows. 

In explicit terms, the above properties guarantee that two minimal projective presentations represented by non-projective $A$-modules in $W_{\Sigma}$, when seen as arcs on a line or on a circle, respectively, are either completely separated, consecutive or they cover each other properly. The projective $A$-modules in $W_{\Sigma}$ can be seen as uncovered and unattached points in this picture. Moreover, conditions (1) and (2) put restrictions on the length of these arcs as well as on the length of their possible chains. Now it is not hard to see that every set $\Xcal:=\{X_1,...,X_s\}$ of indecomposable $A$-modules (up to isomorphism), fulfilling the above properties, such that every $X_i$ belongs to the top-series of a different indecomposable projective $A$-module, equals the set $W_{\Xcal}$, induced by the universal localisation $A_{\Xcal}$.

Next, we will modify $W_{\Sigma}$ to get another set $\tilde{W}_{\Sigma}$ of indecomposable $A$-modules.
In fact, whenever there is a maximal subset $\{X_{i_j}\}\subseteq W_{\Sigma}$ such that the minimal projective presentations of the pairwise different $X_{i_j}$ form a non-trivial chain of the form
$$\sigma_0^{X_{i_1}}\circ...\circ\sigma_0^{X_{i_l}}=:\sigma^*$$
for $j\in\{1,...,l\}$, we replace $X_{i_1}$ by $\tilde{X}_{i_1}:=cok(\sigma^*)$ and $X_{i_j}$, for $j\not= 1$, by $\tilde{X}_{i_j}:=P_0^{X_{i_j}}$, the projective cover of $X_{i_j}$ in $A\mbox{-}mod$. In other words, we replace a maximal chain of consecutive morphisms, each represented by a non-projective $A$-module in $W_{\Sigma}$, by a long composition and we add indecomposable projective $A$-modules in-between, which no longer belong to $^*\!\Xcal_{A_{\Sigma}}$. Note that all non-projective $\tilde{X}_i$ in $\tilde{W}_{\Sigma}$ still belong to $^*\!\Xcal_{A_{\Sigma}}$. Also, we can get back $W_{\Sigma}$ from $\tilde{W}_{\Sigma}$ by reversing the above process (splitting long morphisms that factor through projective $A$-modules in $\tilde{W}_{\Sigma}$) and we conclude that the universal localisation $A_{\Sigma}$ is uniquely determined by the set $\tilde{W}_{\Sigma}$. Moreover, $\tilde{W}_{\Sigma}$ fulfils the following properties, induced by $W_{\Sigma}$:
\begin{enumerate}
\item[(1')] For all non-projective $X$ in $\tilde{W}_{\Sigma}$, we have $l(X)\le n-1$.
\item[(2')] Minimal projective presentations of modules in $\tilde{W}_{\Sigma}$ can never be composed.
\item[(3')] The minimal projective presentations of two different non-projective modules in $\tilde{W}_{\Sigma}$ cannot have the same domain.
\item[(4')] The minimal projective presentation of a non-projective $A$-module in $\tilde{W}_{\Sigma}$ factors properly through the projective cover $P_0^X$ of a non-projective $A$-module $X$ in $\tilde{W}_{\Sigma}$ if and only if it factors through $P_1^X$.
\end{enumerate}
In explicit terms, the minimal projective presentations represented by the $A$-modules in $\tilde{W}_{\Sigma}$ are either completely separated or they cover each other properly, when seen as arcs and loops on a line or on a circle, respectively. By $\bf{W}$ we denote the set of all isomorphism classes of sets $\{X_1,...,X_s\}$ of indecomposable $A$-modules with $s\le n$, fulfilling the properties (1'), (2'), (3') and (4'), where every $X_i$ belongs to the top-series of a different indecomposable projective $A$-module. We get a bijection
$$uniloc(A)\rightarrow\bf{W}$$
by mapping a universal localisation $A_{\Sigma}$ to $\tilde{W}_{\Sigma}$. We already stated injectivity. Surjectivity follows from reversing the idea on how to pass from $W_{\Sigma}$ to $\tilde{W}_{\Sigma}$ and previous observations.

It remains to prove the bijective correspondence between $\bf{W}$ and the isomorphism classes of all orthogonal collections in $A\mbox{-}mod$. We consider a bijection $\Phi$ on $A\mbox{-}ind$ given by mapping an indecomposable projective $A$-module $P$ to its simple top and an indecomposable non-projective $A$-module $P/rad^tP$ to $P/rad^{t+1}P$ for $1\le t<l(P)$. We claim that $\Phi$ induces a bijection between $\bf{W}$ and the isomorphism classes of all orthogonal collections in $A\mbox{-}mod$ by mapping $\{X_1,...,X_s\}$ in $\bf{W}$ to $\{\Phi(X_1),...,\Phi(X_s)\}$. Let us first check that the assignment yields a well-defined map. Clearly, the empty set in $\bf{W}$ corresponds to the trivial orthogonal collection. Moreover, using property (1'), we know that the length of the $\Phi(X_i)$ is bounded by $n$ such that $End_A(\Phi(X_i))$ is isomorphic to $\mathbb{K}$. Now let $Q_A$ be the underlying quiver of $A$ and, without loss of generality, we assume that $A$ and, thus, $Q_A$ is connected. We number the vertices of $Q_A$ from 1 to $n$. The $z$-th entry of the dimension vector of $\Phi(X_i)$ is given as follows:
$$(dim\Phi(X_i))_z=\left\{\begin{array}{cl} 1, & \mbox{if }\sigma_0^{X_i}:P_1^{X_i}\rightarrow P_0^{X_i} \mbox{ factors through } P_z\\ 0, & \mbox{else} \end{array}\right.$$
Now consider $Hom_A(\Phi(X_i),\Phi(X_j))$ for $i\not= j$. Keeping in mind the shape of the dimension vector, by property (2') and (3'), we know that there cannot be any injective or surjective maps from $\Phi(X_i)$ to $\Phi(X_j)$. Orthogonality finally follows from property (4'). Consequently, $\Phi$ induces a well-defined map from $\bf{W}$ to the set of all isomorphism classes of orthogonal collections in $A\mbox{-}mod$. Moreover, this map is injective, since $\Phi$ is a bijection on $A\mbox{-}ind$. It remains to prove surjectivity. Take an arbitrary orthogonal collection $\Xcal:=\{X_{1},...,X_{s}\}$ in $A\mbox{-}mod$. Clearly, every $X_{i}$ belongs to the top-series of a different indecomposable projective $A$-module and we have $s\le n$. Now we apply the obvious inverse $\Phi^{-1}$ of $\Phi$ to get the set 
$$\Phi^{-1}(\Xcal):=\{\Phi^{-1}(X_{1}),...,\Phi^{-1}(X_{s})\}$$ 
of indecomposable $A$-modules. We have to show that $\Phi^{-1}(\Xcal)$ belongs to $\bf{W}$. Since $End_A(X_{i})\cong\mathbb{K}$, we know that $\Phi^{-1}(\Xcal)$ fulfils property (1'). The properties (2'), (3') and (4') follow from the orthogonality of the $X_{i}$. This finishes the proof.
\end{proof}

We have the following immediate corollaries.

\begin{corollary}\label{locNak}
Let $A$ be a Nakayama algebra. There is a bijection
$$\omega:uniloc(A)\longrightarrow wide(A)$$
by mapping a universal localisation $A_{\Sigma}$ to $\Xcal_{A_{\Sigma}}=\Sigma^*$. The inverse is given by mapping a wide subcategory $\Ccal$ of $A\mbox{-}mod$ to $A_{\,^*\!\Ccal}$. Indeed, a ring epimorphism $A\rightarrow B$ is a universal localisation if and only if $Tor_1^A(B,B)=0$.
\end{corollary}

\begin{proof}
Combining Proposition \ref{Prop Nak} and Theorem \ref{Main Nak}, we get a bijective correspondence between the epiclasses of the universal localisations of $A$ and the wide subcategories of $A\mbox{-}mod$. However, it follows from the construction that this bijection is not given by $\omega$. But since both of the sets are finite and since, by Section \ref{UNIFDA}, $\omega$ already defines an injective map, we are done. 
\end{proof}

\begin{corollary}\label{cor pic}
Consider the Nakayama algebras $A_n^h$ and $\tilde{A}_n^h\,$ for $h>1$. The following holds.
\begin{enumerate}
\item There is a bijective correspondence between the epiclasses of the universal localisations of $A_n^h$ and the possible configurations of non-crossing arcs with length at most the minimum of $\{n-1,h-1\}$ on a line with $n$ linearly ordered points. By convention, a loop has length zero.
\item There is a bijective correspondence between the epiclasses of the universal localisations of $\tilde{A}_n^h$ and the possible configurations of non-crossing arcs with length at most the minimum of $\{n-1,h-1\}$ on a circle with $n$ linearly ordered points. Again, a loop is considered to have length zero.
\end{enumerate}
\end{corollary}

\begin{proof}
On the one hand, part two of the corollary can be deduced from \cite{D} (see Corollary 2.6.12) combined with Theorem \ref{Main Nak}. On the other hand, the whole statement follows from a careful analysis of the proof of Theorem \ref{Main Nak}. More precisely, the set $\bf{W}$ in the proof corresponds naturally to the wanted set of configurations of non-crossing arcs. Indeed, for a fixed set of indecomposable $A$-modules $\Xcal$ in $\bf{W}$ we draw an arc from $j$ to $i$ (respecting the given orientation on the points), whenever the cokernel of the map $P_i\rightarrow P_j$ belongs to $\Xcal$. Moreover, a projective $A$-module $P_k$ in $\Xcal$ gives rise to a loop at the point $k$.
\end{proof}

The previous discussion allows us to count the universal localisations of $A_n^h$ and $\tilde{A}_n^h$. In \cite{D} (Section 2.6), this was done with respect to the orthogonal collections in the module category. Note that for $\tilde{A}_n^h$ with $h\ge n$ its number is given by {\tiny$\left(\begin{array}{c} 2n \\ n \end{array}\right)$}, independent of the choice of $h$.

\subsection{Homological ring epimorphisms for Nakayama algebras}\label{NAKUNI2}
By Corollary \ref{locNak}, we already know that all homological ring epimorphisms of $A$ are universal localisations. But the converse is far from being true, as the following example illustrates.

\begin{example}
Consider the Nakayama algebra $A:=A_n^h$ for $n\ge 3$ and $2\le h<n$. Take $x$ and $r<h$ in $\mathbb{N}_0$ with $n=xh+r$. By \cite{HS} (see Proposition 1.4), the global dimension of $A$ is given by
$$gldim(A)=pd(S_1)=\left\{\begin{array}{cl} 2x-1, & \mbox{if } r=0 \\ 2x, & \mbox{if } r=1\\ 2x+1, & \mbox{else} \end{array}\right.$$
Now let $P$ be the projective $A$-module appearing last in the minimal projective resolution of $S_1$. Note that $P$ is indecomposable and we have $P\not= P_1,P_2$. By $\Ccal_{S_1}$ we denote the wide and semisimple subcategory $add\{P,S_1\}$ of $A\mbox{-}mod$. By construction, we clearly have $Ext_A^i(S_1,P)\not= 0$ for some $i>1$. It follows that the universal localisation $A_{\,^*\!\Ccal_{S_1}}$ is not homological, since for all $i\ge 1$ we have,
$$Ext_{A_{\,^*\!\Ccal_{S_1}}}^i\!\!\!(S_1,P)=0.$$
In the minimal case for $n=3$ and $h=2$ the universal localisation $A_{\,^*\!\Ccal_{S_1}}$, here given by $A/Ae_2A$, is the unique universal localisation of $A$ not yielding a homological ring epimorphism. In general, we get plenty of those, e.g., by localising at certain subsets of $\,^*\!\Ccal_{S_1}$. Indeed, the ring epimorphism $A\rightarrow A/Ae_2A$ and for $n>3$ the ring epimorphisms $A\rightarrow A/Ae_3A$ and $A\rightarrow A/A(e_2+e_3)A$ are never homological.
\end{example}

Next, we want to discuss and classify the homological ring epimorphisms for self-injective Nakayama algebras. By Corollary \ref{cor hom}, we already know that they are precisely given by those ring epimorphisms $f:A\rightarrow B$ which turn $B$ into a projective $A$-module. We are now looking for a more explicit description. For the rest of this section consider $A$ to be $\tilde{A}_n^h$ for $n,h\ge 2$ and let $M$ be a non-projective and indecomposable $A$-module. Indeed, $M$ is of infinite projective dimension and periodic with respect to the syzygy-functor $\Omega_A$. The following lemma describes this periodicity and the corresponding $Ext$-groups.

\begin{lemma}\label{lem hom}
Let $A$ and $M$ be as above. Denote by $s$ the length of the $A$-module $M$. The following holds.
\begin{enumerate}
\item If $s\not= \frac{1}{2}h$, we have $\Omega_A^{z}(M)=M$ if and only if $z=2x\,$ for $x\in\mathbb{N}_{\ge 1}$ with $xh\equiv 0$ modulo $n$. If these equivalent conditions hold, we get $Ext_A^z(M,M)\cong\mathbb{K}\,$.
\item If $s=\frac{1}{2}h$, we have $\Omega_A^{z}(M)=M$ if and only if $z\in\mathbb{N}_{\ge 1}$ with $zs\equiv 0$ modulo $n$. If these equivalent conditions hold, we get $Ext_A^z(M,M)\cong\mathbb{K}\,$.
\end{enumerate}
\end{lemma}

\begin{proof}
Note that, since $M$ is not projective, we have $s<h$.
Assume that $M$ belongs to the top-series of the indecomposable projective $A$-module $P_i$ for $i\in\{1,...,n\}$. The minimal projective resolution of $M$ is of the following form
$$\xymatrix{\Pcal_M:  ...\ar[r] & P_{\overline{i+2h}}\ar[r] & P_{\overline{i+h+s}}\ar[r] & P_{\overline{i+h}}\ar[r] & P_{\overline{i+s}}\ar[r] & P_i\ar[r] & M}$$
where the indices must be read modulo $n$. By convention, we define $P_0$ to be $P_n$. Now if $s\not= \frac{1}{2}h$, then the length of the indecomposable $A$-module $\Omega_A^z(M)$, for $z\in\mathbb{N}_{\ge 1}$ odd, is $h-s\not= s$ such that $\Omega_A^z(M)\not= M$. On the other hand, the length of $\Omega_A^{z}(M)$ always equals $s$, if $z$ is even. Hence, keeping in mind that every indecomposable $A$-module is uniserial, we get 
$$\Omega_A^{z}(M)=M\Longleftrightarrow z\in 2\mathbb{N}_{\ge 1}\,\wedge\,\frac{1}{2}zh\equiv 0\,\,\,\text{mod}\,\, n.$$ 
If $s=\frac{1}{2}h$, then the length of the indecomposable $A$-module $\Omega_A^z(M)$ for $z\in\mathbb{N}_{\ge 1}$ always equals $s$ and, thus, 
$$\Omega_A^{z}(M)=M\Longleftrightarrow zs\equiv 0\,\,\,\text{mod}\,\,n.$$ 
Moreover, if $\Omega_A^{z}(M)=M$, we clearly have that $Ext_A^z(M,M)\not= 0$. A careful analysis of $Hom_A(\Pcal_M,M)$ yields $$Ext_A^z(M,M)\cong\mathbb{K}\,.$$ 
\end{proof}

Now let $f:A\rightarrow B$ be a homological ring epimorphism. We call $f$ \textbf{semisimple}, if $B$ is a semisimple $\mathbb{K}$-algebra. Since, by Corollary \ref{cor hom}, $B$ is a projective $A$-module, for $A=\tilde{A}_n^h$ with $h>n$ there are no non-zero semisimple homological ring epimorphisms. Indeed, for $h>n$ we have non-trivial endomorphisms for every indecomposable projective $A$-module. Moreover, for $h\le n$ the semisimple homological ring epimorphisms of $A$ are classified by the possible orthogonal collections of indecomposable projective $A$-modules in $A\mbox{-}mod$. From now on, we will assume that $f$ is not semisimple. Let $M$ be a non-projective $A$-module in $\Xcal_B$ and $z\in\mathbb{N}_{\ge 1}$ minimal such that $\Omega_A^{z}(M)=M$. Since $f$ is homological, by Lemma \ref{lem hom}, it follows that
$$Ext_B^z(M,M)\cong Ext_A^z(M,M)\cong\mathbb{K}\, .$$
Hence, the minimal projective resolution $\Pcal_M$ of $_AM$ is contained in $\Xcal_B$ and coincides with the minimal projective resolution of $M$ as a $B$-module (this can also be deduced from the fact that $_AB$ is a projective $A$-module). Let $\Ccal_M$ be the smallest additive subcategory of $A\mbox{-}mod$ containing $M$, all indecomposable projective $A$-modules appearing in $\Pcal_M$ and the objects $\Omega_A^r(M)$ for $r>0$. Then $\Ccal_M$ is the smallest (not necessarily abelian) higher extension-closed subcategory of $A\mbox{-}mod$ containing $M$ and, clearly, we have $\Ccal_M\subseteq\Xcal_B$. By Lemma \ref{lem hom}, we get the following immediate consequences.
\begin{itemize}
\item If $h=2$, then all non-trivial homological ring epimorphisms $A\rightarrow B$ are semisimple. Note that for $h=2$ all non-projective indecomposable $A$-modules are simple and, by Lemma \ref{lem hom}(2), we get that $\Ccal_S$ already equals $A\mbox{-}mod$ for any simple $A$-module $S$.
\item If $h=n-1$ or if $n$ is a prime number with $h<n$, then a similar analysis of Lemma \ref{lem hom} yields that all non-trivial homological ring epimorphisms $A\rightarrow B$ are semisimple.
\end{itemize}

\begin{example}
In case $h<n$, the first example of a non-trivial and non-semisimple homological ring epimorphism $f:A\rightarrow B$ occurs for $\tilde{A}_6^3$. There are precisely three such choices given by the universal localisations at $\Sigma=\{S_1,S_4\}$, $\Sigma=\{S_2,S_5\}$ or $\Sigma=\{S_3,S_6\}$. In all these cases, the $\mathbb{K}$-algebra $A_{\Sigma}$ is Morita-equivalent to $\tilde{A}_4^2$ and $\Xcal_{A_{\Sigma}}$ is given by $\Ccal_{S_1}$, $\Ccal_{S_2}$ or $\Ccal_{S_3}$, respectively.
\end{example}

Next, we classify the homological ring epimorphisms for (connected) self-injective Nakayama algebras by using the classification of the universal localisations.

\begin{theorem}\label{class hom}
Let $A$ be a self-injective Nakayama algebra of the form $\tilde{A}_n^h\,$ for $n,h\ge 2$.
\begin{enumerate}
\item $A$ admits a non-zero semisimple homological ring epimorphism if and only if $h\le n$. These ring epimorphisms are classified by the possible non-empty orthogonal collections of indecomposable projective $A$-modules.
\item $A$ admits a non-trivial and non-semisimple homological ring epimorphism if and only if $gcd(n,h)=d\not= 1$ and $h>2$. These ring epimorphisms are classified by the orthogonal collections of simple $A$-modules of the form $\{S_{i_1},...,S_{i_k}\}$ with $i_j\in\{1,...,d\}$ pairwise different and 
$$k\in\left\{\begin{array}{cl} \{1,...,d-1\}, & \mbox{if } h\not= d \\ \{1,...,d-2\}, & \mbox{if } h=d. \end{array}\right.$$
\end{enumerate}
\end{theorem}

\begin{proof}
ad(1): Follows from Corollary \ref{cor hom} and the fact that indecomposable projective $A$-modules have trivial endomorphism algebras if and only if $h\le n$. ad(2): Let $f:A\rightarrow B$ be a non-trivial and non-semisimple homological ring epimorphism. By previous arguments, we already know that $h$ must be greater than $2$. Combining Corollary \ref{cor hom} and Lemma \ref{lem Nak1}, we know that $B$ is a self-injective Nakayama algebra and that $_AB$ is a projective $A$-module. Consequently, since $A$ is connected and $B$ is not semisimple, also $B$ is connected and, thus, (up to Morita-equivalence) of the form $\tilde{A}_{\tilde{n}}^{\tilde{h}}$ for $2\le \tilde{n}\le n$ and $2\le \tilde{h}\le h$. By Corollary \ref{cor hom} and Corollary \ref{locNak}, we can write $f$ as a universal localisation $f:A\rightarrow A_{\Sigma_B}$. We will consider the set $W_{\Sigma_B}$, which determines the localisation (see Lemma \ref{lemNak}). We claim that $W_{\Sigma_B}$ only contains simple $A$-modules. First of all, since $B$ is of infinite global dimension, $W_{\Sigma_B}$ cannot contain any projective $A$-modules. Now suppose there is an indecomposable $A$-module $M$ in $W_{\Sigma_B}$ with $1<l(M)=s<min\{n,h\}$ (recall that a non-projective module in $W_{\Sigma_B}$ has length at most $n-1$). Without loss of generality, we can choose $M$ to be of minimal length among the non-simple $A$-modules in $W_{\Sigma_B}$. Let $P_i$ for $i\in\{1,...,n\}$ be the projective cover of $M$ in $A\mbox{-}mod$ and we choose $j\in\{1,...,n\}$ such that the corresponding simple $A$-modules fulfil $Ext_A^1(S_i,S_j)\not= 0$. Now it can be checked easily that the $A$-module $P_j/rad^{s-1}P_j$ - the radical of $M$ - belongs to $\Xcal_{A_{\Sigma_B}}$ (compare Lemma \ref{lemNak} and the defining properties for $W_{\Sigma_B}$ discussed in the proof of Theorem \ref{Main Nak}). But the projective $A$-module $P_j$ does not carry an $A_{\Sigma_B}$-module structure, since $Hom_A(\sigma_0^M,P_{j})$ is not an isomorphism. This yields a contradiction, keeping in mind that the $A$-module $_AB$ must be projective. Consequently, $W_{\Sigma_B}$ contains only simple $A$-modules or, equivalently, $A_{\Sigma_B}$ is given by inverting certain arrows in the underlying quiver $\tilde{\Delta}_n$. Now the fact that $A_{\Sigma_B}$ is Morita-equivalent to an algebra $\tilde{A}_{\tilde{n}}^{\tilde{h}}$ for $2\le \tilde{n}\le n$ and $2\le \tilde{h}\le h$ induces some periodicity of length $2\le d\le min\{h,n\}$ on the simple modules in $W_{\Sigma_B}$, where $d$ divides $h$ and $n$. More precisely, $W_{\Sigma_B}$ is determined by a subset of the form $\{S_{i_1},...,S_{i_k}\}$ for $i_j\in\{1,...,d\}$ pairwise different and 
$$k\in\left\{\begin{array}{cl} \{1,...,d-1\}, & \mbox{if } h\not= d \\ \{1,...,d-2\}, & \mbox{if } h=d \end{array}\right.$$
such that a simple $A$-module $S_m$ belongs to $W_{\Sigma_B}$ if and only if there is some $j\in\{1,...,k\}$ with $m\equiv i_j$ modulo $d$. Note that we can choose $d$ to be $gcd(n,h)$. In particular, we get $gcd(n,h)\not= 1$. Conversely, if $d>1$ is the greatest common divisor of $h>2$ and $n$, it is easy to check that every universal localisation at a set of simple $A$-modules $\Scal$, admitting a periodicity like above with respect to $d$, yields a non-trivial and non-semisimple homological ring epimorphism $A\rightarrow A_{\Scal}$.
\end{proof}

Note that this result allows us to count the homological ring epimorphisms for self-injective Nakayama algebras (up to epiclasses). For example, take $A$ to be the algebra $\tilde{A}_n^h$ for $n=h\ge 2$. Then there are precisely {\Small$\sum\limits_{i=0}^{n-1}$}{\tiny$\left(\begin{array}{c} n \\ i \end{array}\right)$} non-zero homological ring epimorphisms out of a total number of {\tiny$\left(\begin{array}{c} 2n \\ n \end{array}\right)$} universal localisations. The algebras $\tilde{A}_n^h$ for $h>n$ and $gcd(n,h)=1$ do not admit a non-trivial homological ring epimorphism.

\section{\texorpdfstring{$\tau$}{tau}-tilting modules and universal localisations for Nakayama algebras}\label{NAKTILT}
In this section we will prove a similar result to Theorem \ref{hered tilt} for Nakayama algebras, now using $\tau$-tilting modules. Let $A$ be a Nakayama algebra. The first step will be to compare the torsion classes and the wide subcategories in $A\mbox{-}mod$.
By \cite{IT} (Proposition 2.12), we know that for any $\Tcal$ in $tors(A)$ the subcategory 
$$\alpha(\Tcal):=\{X\in\Tcal\mid \forall(g:Y\rightarrow X)\in\Tcal,\, ker(g)\in\Tcal\}$$
forms an exact abelian and extension-closed, thus wide, subcategory of $A\mbox{-}mod$. Note that, in contrast to the hereditary case (see Theorem \ref{Thm TI}), the split-projective $A$-modules in $\Tcal$ do not necessarily belong to $\alpha(\Tcal)$.  
We want to show that $\alpha$ yields a bijection between
$$tors(A)\longrightarrow wide(A).$$
We have to construct an inverse to $\alpha$. Let $\Ccal$ be a wide subcategory of $A\mbox{-}mod$. Note that, again in contrast to the hereditary case, $Gen\Ccal$ is not, in general, closed under extensions. Consequently, we set
$$\beta(\Ccal):=add\{X\in A\mbox{-}ind\mid X\,\,\text{is an extension of modules in}\,\,\,Gen\Ccal\}.$$
Using Lemma \ref{Nakext}, one can check that $\beta(\Ccal)$ describes precisely the subcategory of $A\mbox{-}mod$ containing \textit{all} modules that can be written as an extension of modules in $Gen\Ccal$. The next lemma justifies the definition.

\begin{lemma}\label{torsion}
Let $\Ccal$ be a wide subcategory and $\Tcal$ be a torsion class in $A\mbox{-}mod$. Then the following holds.
\begin{enumerate}
\item $\beta(\Ccal)$ is the smallest torsion class in $A\mbox{-}mod$ containing $\Ccal$.
\item $\alpha(\Tcal)=\alpha(\Tcal)_{ind}:=\{X\in\Tcal\mid \forall(g:Y\rightarrow X)\in\Tcal\,\text{and}\,\, Y\, \text{indecomposable},\, ker(g)\in\Tcal\}$.
\item Every split-projective $A$-module in $\Tcal$ admits a quotient in $\alpha(\Tcal)$.
\end{enumerate}
\end{lemma}

\begin{proof}
ad(1): We will first show that $\beta(\Ccal)$ is closed under quotients. Thus, take $X$ in $\beta(\Ccal)$ and a surjection $f:X\twoheadrightarrow X'$ in $A\mbox{-}mod$. We have to show that $X'$ belongs to $\beta(\Ccal)$. We can assume that $X$ and $X'$ are indecomposable.
Consider the short exact sequence
$$\xymatrix{0\ar[r]& Y\ar[r]^i & X\ar[r]^{\pi} & Z\ar[r] & 0}$$
where $Y$ and $Z$ are indecomposable $A$-modules in $Gen\Ccal$. If $f$ factors through $\pi$, $X'$ belongs to $Gen\Ccal\subseteq\beta(\Ccal)$, since $Gen\Ccal$ is closed under quotients. Otherwise, since every indecomposable $A$-module is uniserial, $\pi$ factors through $f$ and we can consider the following commutative diagram of indecomposable $A$-modules
$$\xymatrix{ & 0\ar[d] & & & \\ & ker(f\circ i)\ar[d]^{i'}& & & \\ 0\ar[r]& Y\ar[r]^i\ar[d] & X\ar[r]^{\pi}\ar[d]^f & Z\ar[r]\ar[d]^{id} & 0 \\ 0\ar[r]& cok(i')=ker(\pi')\ar[r] & X'\ar[r]^{\pi '} & Z\ar[r] & 0}$$
Since $Gen\Ccal$ is closed under quotients, $ker(\pi ')$ belongs to $Gen\Ccal$ and $X'$ can be written as an extension of modules in $Gen\Ccal$. Hence, $X'$ lies in $\beta(\Ccal)$.

Next, we want to see that $\beta(\Ccal)$ is closed under extensions. We start with a general statement about the structure of a module in $\beta(\Ccal)$. This observation will be crucial in the actual proof afterwards. Let $X$ be an indecomposable $A$-module in $\beta(\Ccal)$ together with a short exact sequence
$$\xymatrix{0\ar[r]& Y\ar[r]^i & X\ar[r]^{\pi} & Z\ar[r] & 0}$$
where $Y$ and $Z$ are again indecomposable $A$-modules in $Gen\Ccal$. Let $C_Y$ and $C_Z$ be indecomposable $A$-modules in $\Ccal$ surjecting onto $Y$ and $Z$, respectively. Since every indecomposable $A$-module is uniserial, either the $A$-module $X$ belongs to $Gen\Ccal$ or the map $\pi$ factors through $C_Z$ such that $Z$ belongs to $\Ccal$, as the cokernel of the induced map from $C_Y$ to $C_Z$. We call this property $(*)$.

Now let $X_1$ and $X_2$ be two indecomposable $A$-modules in $\beta(\Ccal)$. By Lemma \ref{Nakext}, a non-trivial extension of these two modules is of the form
$$\xymatrix{& & V\ar@{>>}[dr] & & \\ 0\ar[r] & X_1\ar@{^{(}->}[ur]^j\ar@{>>}[dr] & & X_2\ar[r] & 0.\\ & & W\ar@{^{(}->}[ur]^k & &}$$
Since $\beta(\Ccal)$ is closed under quotients, $W$ and $cok(j)=cok(k)$ belong to $\beta(\Ccal)$. It suffices to show that $V$ belongs to $\beta(\Ccal)$. We will consider two different cases with respect to the following short exact sequence

$$\xymatrix{0\ar[r]& X_1\ar[r]^j & V\ar[r] & cok(j)\ar[r] & 0.}$$

\underline{Case1}: Assume that $cok(j)$ lies in $Gen\Ccal$.
First of all, if also $X_1$ belongs to $Gen\Ccal$, we are done by the definition of $\beta(\Ccal)$. Otherwise, by the property $(*)$, we get a short exact sequence of the form

$$\xymatrix{0\ar[r]& Y\ar[r] & X_1\ar[r] & Z\ar[r] & 0}$$
with $Y$ indecomposable in $Gen\Ccal$ and $Z$ indecomposable in $\Ccal$, yielding the following induced exact sequence

$$\xymatrix{0\ar[r]& Y\ar[r]^{j'} & V\ar[r] & cok(j')\ar[r] & 0.}$$
If now $cok(j')$ belongs to $Gen\Ccal$, we are done by the definition of $\beta(\Ccal)$. Otherwise, using that $cok(j)$ lies in $Gen\Ccal$, there is an indecomposable $A$-module $V'$ in $\Ccal$, fitting into the following commutative diagram

$$\xymatrix{0\ar[r]& X_1\ar[r]^j\ar[d] & V\ar[r]\ar[d] & cok(j)\ar[r]\ar[d]_{id} & 0 \\ 0\ar[r]&  Z\ar[r] & cok(j')\ar[r]\ar[d] & cok(j)\ar[r] & 0\\ & & V'\ar[ur]\ar[d] & &\\ & & 0 & &}$$
Consequently, $cok(j)$ equals the cokernel of the induced map from $Z$ to $V'$ and, thus, it belongs to $\Ccal$. Since $\Ccal$ is closed under extensions, this also forces $cok(j')$ to lie in $\Ccal$, leading to a contradiction. 

\underline{Case2}: Assume that $cok(j)$ does not lie in $Gen\Ccal$. By the property $(*)$, there is a short exact sequence

$$\xymatrix{0\ar[r]& Y_j\ar[r] & cok(j)\ar[r] & C_{j}\ar[r] & 0}$$
with $Y_j$ indecomposable in $Gen\Ccal$ and $C_{j}$ indecomposable in $\Ccal$, yielding the following commutative diagram

$$\xymatrix{0\ar[r]& ker(s)\ar[r]\ar[d]^{p} & V\ar[r]^{s}\ar[d] & C_j\ar[r]\ar[d]^{id} & 0 \\ 0\ar[r]&  Y_j\ar[r] & cok(j)\ar[r] & C_j\ar[r] & 0}$$
with surjective vertical morphisms.
Now either $ker(s)$ belongs to $Gen\Ccal$ and, thus, $V$ lies in $\beta(\Ccal)$, as wanted, or, the map $p$ must factor through some indecomposable $A$-module $C_{Y_j}$ in $\Ccal$, since $Y_j$ belongs to $Gen\Ccal$. In the second case, we get a short exact sequence of the form

$$\xymatrix{0\ar[r]& C_{Y_j}\ar[r] & E\ar[r] & C_j\ar[r] & 0}$$
where $E$ is an indecomposable $A$-module in $\Ccal$ surjecting onto $cok(j)$. This contradicts our assumption that $cok(j)$ does not lie in $Gen\Ccal$. Consequently, $\beta(\Ccal)$ forms a torsion class in $A\mbox{-}mod$.
Moreover, by construction, $\beta(\Ccal)$ is the smallest torsion class in $A\mbox{-}mod$ containing $\Ccal$.

ad(2): Clearly, we have $\alpha(\Tcal)\subseteq\alpha(\Tcal)_{ind}$. Conversely, take $X$ in $\alpha(\Tcal)_{ind}$, $Y$ in $\Tcal$ and a map $g:Y\rightarrow X$. We have to show that $ker(g)$ belongs to $\Tcal$. To begin with, we can assume $X$ to be indecomposable, since $\alpha(\Tcal)_{ind}$ is closed under direct summands. In particular, the image of $g$ is an indecomposable $A$-module. Moreover, without loss of generality, we can assume that also the kernel of $g$ is indecomposable and that $g$ is not a split map. It follows, by Lemma \ref{Nakext}, that $Y=Y_1\oplus Y_2$ for $Y_1$ and $Y_2$ indecomposable and that $g$ is induced by $(\pi,i)$, like in the following commutative diagram of indecomposable $A$-modules
$$\xymatrix{ & & Y_1\ar@{>>}[dr]^{\pi} & \\ & ker(g)\ar@{^{(}->}[ur]\ar@{>>}[dr] &  & Im(g)\\ ker(\pi)\ar@{^{(}->}[ur]^{i'} & & Y_2=cok(i')\ar@{^{(}->}[ur]^i &}$$
Since $X$ lies in $\alpha(\Tcal)_{ind}$, it follows that $ker(\pi)$ belongs to $\Tcal$ and, thus, $ker(g)$ can be written as an extension of modules in $\Tcal$. Therefore, $ker(g)$ belongs to $\Tcal$.

ad(3): It is enough to show the statement for an indecomposable split-projective $A$-module $T$ in $\Tcal$. If $T$ belongs to $\alpha(\Tcal)$, we are done. Now assume that $T\notin\alpha(\Tcal)$. Suppose that $GenT\cap\alpha(\Tcal)=\{0\}$. Consequently, one can check inductively that all submodules of $T$ in $A\mbox{-}mod$ do not belong to $\Tcal$. Hence, maps to $T$ in $\Tcal$ are trivial such that $T$ lies in $\alpha(\Tcal)$, a contradiction.
\end{proof}

The following proposition establishes the wanted bijection.

\begin{proposition}\label{torsNak}
Let $A$ be a Nakayama algebra. There is a bijection between
$$tors(A)\longrightarrow wide(A)$$
by mapping a torsion class $\Tcal$ to $\alpha(\Tcal)$. The inverse of $\alpha$ is given by $\beta$.
\end{proposition}
\begin{proof}
We will first show that for $\Ccal$ in $wide(A)$ we have $\alpha(\beta(\Ccal))=\Ccal$.

ad$"\supseteq"$: Take $C$ in $\Ccal$ indecomposable, $X$ in $\beta(\Ccal)$ and a map $f:X\rightarrow C$. We have to check that $ker(f)$ lies in $\beta(\Ccal)$.
Using Lemma \ref{torsion}(2), we can assume that $X$ is indecomposable. First of all, if $X$ belongs to $Gen\Ccal$, we are done, since there is an indecomposable $A$-module $C_X$ in $\Ccal$ surjecting onto $X$ such that the kernel of the induced map from $C_X$ to $C$ forces the kernel of $f$ to lie in $Gen\Ccal\subseteq\beta(\Ccal)$. Otherwise, by the property $(*)$ in the proof of Lemma \ref{torsion}(1), we have a short exact sequence of the form
$$\xymatrix{0\ar[r]& Y\ar[r] & X\ar[r]^{\pi} & Z\ar[r] & 0}$$
with $Y$ indecomposable in $Gen\Ccal$ and $Z$ indecomposable in $\Ccal$. First assume that $\pi$ factors through $Im(f)$. Consequently, we get the following commutative diagram of indecomposable $A$-modules
$$\xymatrix{ & C_Y\ar[r]^g\ar@{>>}[d] & X\ar[r]^f\ar@{>>}[d] & C\\ ker(f\circ g)\ar@{^{(}->}[ur]\ar@{>>}[d] & Y\ar@{^{(}->}[ur] & Im(f)\ar@{>>}[d]\ar@{^{(}->}[ur] & \\ ker(f)\ar@{^{(}->}[ur] & & Z &}$$
It follows that $ker(f)$ belongs to $Gen\Ccal\subseteq\beta(\Ccal)$.
Otherwise, we get the following commutative diagram
$$\xymatrix{ & & X\ar@{>>}[d]^{\pi} \\ & ker(f)\ar@{^{(}->}[ur]\ar@{>>}[d] & Z\ar@{>>}[d]^{\pi'} \\ Y\ar@{^{(}->}[ur] & ker(\pi')\ar@{^{(}->}[ur] & Im(f) }$$
Since the kernel of $\pi'$ belongs to $\Ccal$, as the kernel of the induced map from $Z$ to $C$, it follows that $ker(f)$ is an extension of modules in $Gen\Ccal$ and, thus, it lies in $\beta(\Ccal)$, completing the argument.

ad$"\subseteq"$: Take $X$ in $\alpha(\beta(\Ccal))$ indecomposable and show that it belongs to $\Ccal$.
We first assume that $X$ lies in $Gen\Ccal$, getting the following short exact sequence
$$\xymatrix{0\ar[r]& ker(\pi)\ar[r] & C\ar[r]^{\pi} & X\ar[r] & 0}$$
where $C$ in $\Ccal$ is indecomposable and $ker(\pi)$ belongs to $\beta(\Ccal)$, by assumption. Thus, by the definition of $\beta(\Ccal)$, there is an indecomposable $A$-module $C_{\pi}$ in $\Ccal$ mapping non-trivially to $ker(\pi)$ and, hence, yielding an induced map $g:C_{\pi}\rightarrow C$. Then $\pi$ factors through the cokernel of $g$, which again belongs to $\Ccal$. By repeating the argument with $cok(g)$ instead of $C$, we get, after finitely many steps, that $X$ lies in $\Ccal$.

Now we assume that $X\notin Gen\Ccal$. Since $X$ lies in $\beta(\Ccal)$, we can use the property $(*)$ to get the diagram
\begin{equation}\label{diagram1}\xymatrix{0\ar[r]& Y\ar[r] & X\ar[r] & Z\ar[r] & 0\\ & C_Y\ar@{>>}[u]\ar[ur]_{\psi} & & &}\end{equation}
with $C_Y$ and $Z$ indecomposable in $\Ccal$ and $Y$ indecomposable in $Gen\Ccal$. Since $X$ lies in $\alpha(\beta(\Ccal))$, we get that $ker(\psi)$ belongs to $\beta(\Ccal)$. If $ker(\psi)$ lies, indeed, in $Gen\Ccal$, then the $A$-module $Y$ has to be in $\Ccal$, as the cokernel of a map between indecomposable $A$-modules in $\Ccal$. Thus, also $X$ lies in $\Ccal$, as an extension of modules in $\Ccal$, contradicting our assumption. Otherwise, if we assume that $ker(\psi)$ does not belong to $Gen\Ccal$, we can again use the property $(*)$  to get a similar commutative diagram as before
$$\xymatrix{0\ar[r]& Y_{\psi}\ar[r] & ker(\psi)\ar[r] & C_{\psi}\ar[r] & 0\\ & C_{Y_{\psi}}\ar@{>>}[u]\ar[ur]_{\psi'} & & &}$$
with $C_{\psi}$ and $C_{Y_{\psi}}$ indecomposable in $\Ccal$. By composition, we now get a new map $\phi:C_{Y_{\psi}}\rightarrow C_Y$ such that the morphism $\psi$ factors through $cok(\phi)$ in $\Ccal$. Therefore, we can replace $C_Y$ by $cok(\phi)$ in the diagram (\ref{diagram1})
$$\xymatrix{0\ar[r]& Y\ar[r] & X\ar[r] & Z\ar[r] & 0\\ & cok(\phi)\ar@{>>}[u]\ar[ur]_{\tilde{\psi}} & & &}$$
and repeat the whole argument. After finitely many steps, we conclude that $Y$ and, thus, also $X$ lies in $\Ccal$, again yielding a contradiction.

Next, we have to verify that for $\Tcal$ in $tors(A)$ we have $\beta(\alpha(\Tcal))=\Tcal$.

ad$"\supseteq"$: It suffices to show that all indecomposable split-projective modules in $\Tcal$ belong to $\beta(\alpha(\Tcal))$.
Let $T$ in $\Tcal$ be indecomposable and split-projective. If $T$ belongs to $\alpha(\Tcal)$, we are done. Now assume that $T\notin\alpha(\Tcal)$. By Lemma \ref{torsion}(3), there is an indecomposable $A$-module $X$ in $\alpha(\Tcal)$ yielding the sequence
$$\xymatrix{0\ar[r]& ker(\pi)\ar[r] & T\ar[r]^{\pi} & X\ar[r] & 0}$$
with $ker(\pi)$ in $\Tcal$. If $ker(\pi)$ also belongs to $Gen(\alpha(\Tcal))$, we get that $T$ lies in $\beta(\alpha(\Tcal))$, by definition. Otherwise, if $ker(\pi)$ is not in $Gen(\alpha(\Tcal))$, we can deduce from Lemma \ref{torsion}(3) that there must be an indecomposable $A$-module $X'$ in $\alpha(\Tcal)$ yielding the short exact sequence
$$\xymatrix{0\ar[r]& ker(\pi')\ar[r] & ker(\pi)\ar[r]^-{\pi'} & X'\ar[r] & 0}$$
where $ker(\pi')$ lies in $\Tcal$. If now $ker(\pi')$ also belongs to $Gen(\alpha(\Tcal))$, we get that $ker(\pi)$ is in $\beta(\alpha(\Tcal))$ and, hence, $T$ lies in the torsion class $\beta(\alpha(\Tcal))$. Otherwise, we can repeat the previous argument, until, after finitely many steps, the corresponding kernel must belong to $Gen(\alpha(\Tcal))$. This finishes the argument.

ad$"\subseteq"$: The inclusion holds, since $\alpha(\Tcal)\subseteq\Tcal$ and $\beta(\alpha(\Tcal))$ is, by construction, the smallest torsion class in $A\mbox{-}mod$ containing $\alpha(\Tcal)$, see Lemma \ref{torsion}(1).
\end{proof}

\begin{corollary}\label{Nak bij}
Let $A$ be a Nakayama algebra. There are bijections between the following sets:
\begin{enumerate}
\item isomorphism classes of basic support $\tau$-tilting $A$-modules;
\item torsion classes in $A\mbox{-}mod$;
\item wide subcategories in $A\mbox{-}mod$;
\item isomorphism classes of orthogonal collections in $A\mbox{-}mod$;
\item epiclasses of universal localisations of $A$.
\end{enumerate}
\end{corollary} 

\begin{proof}
The bijection between (1) and (2) follows from Theorem \ref{AIR1}. The correspondences between (3), (4) and (5) are given by Proposition \ref{Prop Nak} and Theorem \ref{Main Nak}. Finally, Proposition \ref{torsNak} finishes the proof.
\end{proof}

\begin{remark}
The presented list of bijections can be extended taking into account further results in \cite{AIR} and \cite{BY}. For example, there are correspondences between support $\tau$-tilting modules and certain silting or cluster tilting objects, (co-)t-structures and g-matrices for a given finite dimensional algebra. Nevertheless, in order to keep notation low, it is convenient for us to focus on the presented objects in the corollary above. We refer to the literature for further directions. 
\end{remark}

In what follows, we explore the correspondence between the support $\tau$-tilting modules and the universal localisations of $A$.

\begin{theorem}\label{Naktilt}
Let $A$ be a Nakayama algebra. 
\begin{enumerate}
\item There is a bijection
$$\Psi_A:s\tau\mbox{-}tilt(A)\longrightarrow uniloc(A)$$
by mapping a support $\tau$-tilting $A$-module $T$ to $A_{A_{\Sigma_T}}\!\!:=A_{\,^*(\alpha(GenT))}$. The inverse is given by mapping a universal localisation $A_{\Sigma}$ to $T_\Sigma$, the sum of the indecomposable Ext-projectives in $\beta(\Sigma^*)$.
\item $\Psi_A$ restricts to a bijection between
$$\tau\mbox{-}tilt(A)\longrightarrow uniloc^p(A).$$
\item $\Psi_A$ restricts to a bijection between
$$s\tau\mbox{-}tilt(A/AeA)\longrightarrow uniloc_e(A)$$
for an idempotent $e$ in $A$. In particular, if $T$ is equivalent to $_A(A/AeA)$, it is mapped to $A_{\Sigma_T}=A/AeA$.
\end{enumerate}
\end{theorem}
\begin{proof}
ad(1): Follows from Theorem \ref{AIR1}, Proposition \ref{torsNak} and Corollary \ref{locNak}.

ad(2): Let $T$ be a basic $\tau$-tilting $A$-module. By \cite{AIR} (Proposition 2.2), $T$ is sincere such that $Hom_A(P,T)\not= 0$ for all indecomposable projective $A$-modules $P$. We have to show that $Hom_A(P,\alpha(GenT))\not= 0$ for all $P$. Now let $T'$ be an indecomposable direct summand of $T$, $P$ be an indecomposable projective $A$-module and $f:P\rightarrow T'$ be a non-trivial morphism. If $T'$ is in $\alpha(GenT)$, there is nothing to show. We distinguish cases with respect to the cokernel of $f$. If $cok(f)$ lies in $Gen(\alpha(GenT))$, we are done, keeping in mind that every indecomposable $A$-module is uniserial and $P$ is projective. Indeed, in the extremal case when $cok(f)$ already belongs to $\alpha(GenT)$, we know that $Im(f)$ lies in $GenT$ and, by Lemma \ref{torsion}(3), we get a surjection from $P$ to an indecomposable module in $\alpha(GenT)$.
As a consequence, again by Lemma \ref{torsion}(3), it remains to consider the case when there is an indecomposable $A$-module $X$ in $\alpha(GenT)$ yielding the following commutative diagram of indecomposable $A$-modules
$$\xymatrix{ P\ar@{>>}[dd]\ar[rr]^f\ar[dr]^{\tilde{f}} & & T'\ar@{>>}[d]^{\pi} \\ & ker(\tilde{\pi}\circ\pi)\ar@{^{(}->}[ur] & cok(f)\ar@{>>}[d]^{\tilde{\pi}}\\ Im(f)\ar@{^{(}->}[ur] & & X}$$
Since $X$ belongs to $\alpha(GenT)$, we know that $ker(\tilde{\pi}\circ\pi)$ lies in $GenT$. Certainly, if $ker(\tilde{\pi}\circ\pi)$ belongs to $Gen(\alpha(GenT))$, we are done, using that $P$ is projective. Otherwise, we can repeat the whole argument with $\tilde{f}$ instead of $f$. Since the length of the indecomposable $A$-module $ker(\tilde{\pi}\circ\pi)$ is smaller than the length of $T'$, after finitely many steps, we get that $Hom_A(P,\alpha(GenT))\not= 0$. Consequently, $P$ does not lie in $\,^*(\alpha(GenT))$ and we have $A_{\Sigma_T}\otimes_A P\not= 0$. It follows that the localisation $A_{\Sigma_T}$ is pure.

Conversely, let $A_{\Sigma}$ be a pure universal localisation of $A$ and let $P$ be an indecomposable projective $A$-module. Now consider the non-trivial $A$-module map
$$\phi_P: P\rightarrow A_{\Sigma}\otimes_A P.$$ 
Since $A_{\Sigma}\otimes_A P$ lies in $\Xcal_{A_{\Sigma}}$, there is a basic split-projective module $T_P$ in $\beta(\Xcal_{A_{\Sigma}})$ surjecting onto $A_{\Sigma}\otimes_A P$. Since $P$ is projective, $\phi_P$ factors through $T_P$ and we get a non-trivial map from $P$ to $T_P$.
Using the fact that $\beta(\Xcal_{A_{\Sigma}})$ is closed under extensions, we conclude that all split-projective modules in $\beta(\Xcal_{A_{\Sigma}})$ are Ext-projective such that $T_P$ becomes a direct summand of $T_{\Sigma}$. Therefore, we get $Hom_A(P,T_{\Sigma})\not= 0$ for all indecomposable projective $A$-modules $P$, telling that $T_{\Sigma}$ is sincere and, thus, by \cite{AIR} (Proposition 2.2), $\tau$-tilting. 

ad(3): From (2) we deduce that for a support $\tau$-tilting $A$-module $T$ and a finitely generated projective $A$-module $P=Ae$ we have $Hom_A(Ae,T)=0$ if and only if $A_{\Sigma_T}\otimes_A Ae=0$. Thus, $T$  belongs to $s\tau\mbox{-}tilt(A/AeA)$ if and only if $A_{\Sigma_T}$ is $e$-annihilating. For the last part of the statement see the proof of Theorem \ref{hered tilt}(3).
\end{proof}

In contrast to the hereditary case, in general, tilting $A$-modules do not arise from universal localisations.

\begin{example}\label{ex not induced}
Consider the Nakayama algebra $A:=A_3^2$ and the tilting $A$-module 
$$T:=P_2\oplus P_1\oplus S_2.$$ 
The associated universal localisation $A_{\Sigma_T}$ of $A$ is given by localising at the $A$-module $S_2$ (see Example \ref{Ex Nak}) and, therefore, the map $f:A\rightarrow A_{\Sigma_T}$ is not monomorphic. Moreover, since the identity map on $A$ is the only monomorphic universal localisation of $A$ (up to epiclasses), $T$ cannot arise from universal localisation.
\end{example}

In the hereditary setting of Section \ref{TULHA}, we compared a basic support tilting $A$-module $T$ directly with its associated universal localisation $A_{\Sigma_T}$ (see Proposition \ref{prop split}). The following proposition initiates a similar comparison in the given Nakayama context. 

\begin{proposition}\label{propsilt}
Let $A$ be a Nakayama algebra and $T$ be a basic support $\tau$-tilting $A$-module.
\begin{enumerate}
\item If $T'$ is an indecomposable direct summand of $T$, then the following are equivalent.
\begin{enumerate}
\item[(i)] $T'$ is not split-projective in $GenT$;
\item[(ii)] $T'$ belongs to $^*\!\Xcal_{A_{\Sigma_T}}$.
\end{enumerate} 
\item If $X$ is an indecomposable $A$-module in $^*\!\Xcal_{A_{\Sigma_T}}$, then the following are equivalent.
\begin{enumerate}
\item[(i)] $X\in addT$;
\item[(ii)] $X\in Gen(\Xcal_{A_{\Sigma_T}})$;
\item[(iii)] $X\in GenT=\beta(\Xcal_{A_{\Sigma_T}})$.
\end{enumerate}
\end{enumerate}
\end{proposition}

\begin{proof}
ad(1): $(i)\Leftarrow (ii):$ If $T'$ lies in $^*\!\Xcal_{A_{\Sigma_T}}$, we have $Hom_A(T',\Xcal_{A_{\Sigma_T}})=0$ and, therefore, by Lemma \ref{torsion}(3), $T'$ cannot be split-projective in $GenT$.\\
$(i)\Rightarrow (ii):$ Now assume that $T'$ is not split-projective in $GenT$. We have to show that $T'$ belongs to $^*\!\!\Xcal_{A_{\Sigma_T}}$. Equivalently, we will show that
\begin{enumerate}[(I)]
\item $Ext_A^1(T',\Xcal_{A_{\Sigma_T}})=0$;
\item $Hom_A(T',\Xcal_{A_{\Sigma_T}})=0$;
\item $Hom_A(\sigma_1^{T'},\Xcal_{A_{\Sigma_T}})=0$.
\end{enumerate}
(I) follows, in particular, from the fact that $T'$ is Ext-projective in $\beta(\Xcal_{A_{\Sigma_T}})$. To prove (II) we suppose that there is an indecomposable $A$-module $X$ in $\Xcal_{A_{\Sigma_T}}$ with a non-trivial map $f:T'\rightarrow X$. By assumption, we get the following commutative diagram of indecomposable $A$-modules
$$\xymatrix{ & \tilde{T}\ar@{>>}[d]\ar[dr]^{\tilde{f}} & \\ ker(\tilde{f})\ar@{^{(}->}[ur]\ar@{>>}[d] & T'\ar@{>>}[d]\ar[r]^{f} & X\\ ker(f)\ar@{^{(}->}[ur] & Im(f)\ar@{^{(}->}[ur] &}$$
where $\tilde{T}$ is split-projective in $\beta(\Xcal_{A_{\Sigma_T}})=GenT$. Since $X$ lies in $\Xcal_{A_{\Sigma_T}}$, we know that $ker(\tilde{f})$ is in $\beta(\Xcal_{A_{\Sigma_T}})$. Consequently, we have $Ext_A^1(T',\beta(\Xcal_{A_{\Sigma_T}}))\not= 0$, a contradiction, since $T'$ is Ext-projective in $\beta(\Xcal_{A_{\Sigma_T}})$.
For (III) we have to show that every map $g:P_1^{T'}\rightarrow X$ for $X$ in $\Xcal_{A_{\Sigma_T}}$ factors through $ker(\pi^{T'})$, where 
$$\xymatrix{P_1^{T'}\ar[r]^{\sigma_0^{T'}} & P_0^{T'}\ar[r]^{\pi^{T'}} & T'\ar[r] & 0}$$ 
describes the minimal projective presentation of $T'$ in $A\mbox{-}mod$. Let us suppose that we have a non-trivial map $g:P_1^{T'}\rightarrow X$ not factoring through $ker(\pi^{T'})$ and with $X$ indecomposable in $\Xcal_{A_{\Sigma_T}}$. Then we get the following commutative diagram of indecomposable $A$-modules
$$\xymatrix{P_1^{T'}\ar@{>>}[d]\ar[r]^g & X\ar@{>>}[d]\ar[r]^{g'} & P_0^{T'}\ar@{>>}[d] \\ Im(g)\ar@{^{(}->}[ur]\ar@{>>}[d] & M\ar@{>>}[d]\ar@{^{(}->}[ur] & T'\ar@{>>}[d]^{\pi'}\\ ker(\pi^{T'})\ar@{^{(}->}[ur] & ker(\pi')\ar@{^{(}->}[ur] & cok(g')}$$
where the $A$-module $M$ lies in $\beta(\Xcal_{A_{\Sigma_T}})$, since $X$ belongs to $\Xcal_{A_{\Sigma_T}}$. Therefore, we get $Ext_A^1(T',\beta(\Xcal_{A_{\Sigma_T}}))\not= 0$, again a contradiction.

ad(2): $(i)\Rightarrow (ii):$ Since $X$ belongs to $^*\!\Xcal_{A_{\Sigma_T}}$, we know that $X$ cannot surject onto any object in $\Xcal_{A_{\Sigma_T}}$. Using that $X\in addT$, we know that $X$ lies in $\beta(\Xcal_{A_{\Sigma_T}})$ and therefore, by Lemma \ref{torsion}(3), we get $X\in Gen(\Xcal_{A_{\Sigma_T}})$. Clearly, we have $(ii)\Rightarrow (iii)$.

$(iii)\Rightarrow (i):$ We have to show that $X$ is Ext-projective in $\beta(\Xcal_{A_{\Sigma_T}})$. Suppose there is an indecomposable $A$-module $M$ in $\beta(\Xcal_{A_{\Sigma_T}})$ with $Ext_A^1(X,M)\not= 0$. Using Lemma \ref{Nakext} and the minimal projective presentation of $X$, we get the following commutative diagram of indecomposable $A$-modules
$$\xymatrix{ & & P_0^X\ar@{>>}[d]\ar@{>>}@/^2pc/[dd]^{\pi^X}\\P_1^X\ar@{>>}[d] & L\ar@{>>}[d]\ar@{^{(}->}[ur] & Y_1\ar@{>>}[d]^{\pi} \\ ker(\pi^X)\ar@{^{(}->}[ur]\ar@{>>}[d] & M\ar@{>>}[d]\ar@{^{(}->}[ur] & X\\ ker(\pi)\ar@{^{(}->}[ur] & Y_2\ar@{^{(}->}[ur] & }$$
(Note that $Y_2$ possibly can be zero. In this situation, we have $ker(\pi^X)=L$ and $ker(\pi)=M$.)
Now since $M$ belongs to $\beta(\Xcal_{A_{\Sigma_T}})$, by Lemma \ref{torsion}(3), there is an indecomposable $A$-module $X_T$ in $\Xcal_{A_{\Sigma_T}}$ such that the projective cover $P_0^M$ of $M$ in $A\mbox{-}mod$ surjects onto $X_T$. Since $X$ lies in $^*\!\Xcal_{A_{\Sigma_T}}$, we have $Hom_A(\sigma_1^X,\Xcal_{A_{\Sigma_T}})=0$ and, therefore, $L$ has to surject onto $X_T$. Moreover, we know that 
$$Ext_A^1(X,\Xcal_{A_{\Sigma_T}})=0$$ 
such that $Y_2$ (or $M$ in case $Y_2$ equals zero) must surject onto $X_T$ via some map $g$ with $ker(g)$ in $\beta(\Xcal_{A_{\Sigma_T}})$. If $Y_2\not= 0$ and $Y_2\not= X_T$, by using again Lemma \ref{torsion}(3), we can repeat the previous argument with $ker(g)$ instead of $M$ to conclude that $ker(g)$ has to surject onto an indecomposable $A$-module in $\Xcal_{A_{\Sigma_T}}$ such that, after finitely many steps, we get that $ker(\pi)$ belongs to $\beta(\Xcal_{A_{\Sigma_T}})$. Note that this conclusion is immediate, if $Y_2=0$ or $Y_2=X_T$. But now $ker(\pi)$ must surject onto some indecomposable $A$-module $X_T'$ in $\Xcal_{A_{\Sigma_T}}$ with $Ext_A^1(X,X_T')\not= 0$, a contradiction, since $Ext_A^1(X,\Xcal_{A_{\Sigma_T}})$ must be zero, by assumption.
\end{proof}

The following theorem allows us to read off completely the associated universal localisation $A_{\Sigma_T}$ from a basic support $\tau$-tilting $A$-module $T$.

\begin{theorem}\label{Nak comp}
Let $A$ be a Nakayama algebra and $T$ be a basic support $\tau$-tilting $A$-module which is $\tau$-tilting over the algebra $A/AeA\,$ for an idempotent $e$ in $A$. Then $A_{\Sigma_T}$ is given by localising at the set $\Sigma_T'$, containing the $A$-module $Ae$ and all non split-projective indecomposable direct summands of $T$ in $GenT$.
\end{theorem}

\begin{proof}
One way of proving the above statement is to show that $^*\!\Xcal_{A_{\Sigma_T}}$ equals $^*\!\Xcal_{A_{\Sigma_T'}}$. By Theorem \ref{Naktilt}(3) and Proposition \ref{propsilt}(1), we know that 
$$^*\!\Xcal_{A_{\Sigma_T'}}\subseteq\, ^*\!\Xcal_{A_{\Sigma_T}}.$$ 
Now take $X$ indecomposable in $^*\!\Xcal_{A_{\Sigma_T}}$. If $X$ belongs to $GenT=\beta(\Xcal_{A_{\Sigma_T}})$, by Proposition \ref{propsilt}(2), it already lies in $addT$ and, thus, using Proposition \ref{propsilt}(1), we get that $X$ lies in $\Sigma_T'$. Moreover, if we consider the minimal projective resolution of $X$ in $A\mbox{-}mod$ and assume that $Hom_A(P_0^X,\Xcal_{A_{\Sigma_T}})=0$, using Theorem \ref{Naktilt}(3), we get that $P_0^X$ is a direct summand of $Ae$ and already belongs to $\Sigma_T'$. Since $X$ lies in $^*\!\!\Xcal_{A_{\Sigma_T}}$, we also get that $Hom_A(P_1^X,\Xcal_{A_{\Sigma_T}})=0$. Again by Theorem \ref{Naktilt}(3), this implies that $P_1^X$ belongs to $\Sigma_T'$ and, therefore, $X$ belongs to $^*\!\Xcal_{A_{\Sigma_T'}}$. Consequently, we can assume that $X$ does not belong to $Gen(\Xcal_{A_{\Sigma_T}})\subseteq GenT$ and that there is an indecomposable $A$-module $X_T$ in $\Xcal_{A_{\Sigma_T}}$ together with a non-trivial map $g:P_0^X\rightarrow X_T$ yielding the following commutative diagram of indecomposable $A$-modules
\begin{equation}\label{diagram2}
\xymatrix{P_1^{X_1}=P_1^X\ar@{>>}[d] & P_1^{X_2}=P_0^X\ar@{>>}[d] & P_0^{X_2}=P_0^{X_1}\ar@{>>}[d] \\ ker(\pi^X)\ar@{^{(}->}[ur]\ar@{>>}[d] & ker(\pi^{X_2})\ar@{>>}[d]\ar@{^{(}->}[ur] & X_T\ar@{>>}[d]\\ ker(\pi^{X_1})\ar@{^{(}->}[ur] & Im(g)\ar@{>>}[d]\ar@{^{(}->}[ur] & X_1\ar@{>>}[d]\\ & X\ar@{^{(}->}[ur]^i & X_2=cok(i)}\end{equation}
Note that, by assumption on $X$, $Hom_A(X,X_T)$ must be zero. Now we will show that the $A$-module $X_T$ can be chosen in a way such that $X_1$ and $X_2$ belong to $^*\!\Xcal_{A_{\Sigma_T}}$. Equivalently, we will show that for $i=1,2$
\begin{enumerate}[(I)]
\item $Hom_A(X_i,\Xcal_{A_{\Sigma_T}})=0$;
\item $Ext_A^1(X_i,\Xcal_{A_{\Sigma_T}})=0$;
\item $Hom_A(\sigma_1^{X_i},\Xcal_{A_{\Sigma_T}})=0$.
\end{enumerate}
(I): Since $X$ is in $^*\!\Xcal_{A_{\Sigma_T}}$, we have $Hom_A(X,\Xcal_{A_{\Sigma_T}})=0$. Thus, every map from $X_1$ to an indecomposable $A$-module in $\Xcal_{A_{\Sigma_T}}$ must factor through $X_2$. Suppose we have a non-trivial map $f:X_2\rightarrow X_T'$ with $X_T'$ indecomposable in $\Xcal_{A_{\Sigma_T}}$. We can lift $f$ to a map $\tilde{f}:X_T\rightarrow X_T'$ such that the inclusion $Im(g)\rightarrow X_T$ factors through $ker(\tilde{f})$. Since $ker(\tilde{f})$ lies in $\Xcal_{A_{\Sigma_T}}$, we can replace $X_T$ by $ker(\tilde{f})$ and repeat the whole argument with different $X_1$ and $X_2$. Since $X$ does not belong to $Gen(\Xcal_{A_{\Sigma_T}})$, we know that $Im(g)$ does not lie in $\Xcal_{A_{\Sigma_T}}$ and, thus, after finitely many steps (I) is fulfilled.
For (II), we first apply the functor $Hom_A(-,X_T')$ for $X_T'$ in $\Xcal_{A_{\Sigma_T}}$ to the short exact sequence
$$\xymatrix{0\ar[r] & X\ar[r]^i & X_1\ar[r] & X_2\ar[r] & 0}$$
to see that it suffices to show that $Ext_A^1(X_2,\Xcal_{A_{\Sigma_T}})=0$. Now suppose that there is an indecomposable $X_T'$ in $\Xcal_{A_{\Sigma_T}}$ with $Ext_A^1(X_2,X_T')\not= 0$. Since we have $Hom_A(X,\Xcal_{A_{\Sigma_T}})=0$, by using Lemma \ref{Nakext}, we conclude that there is a non-trivial morphism from $ker(\pi^{X_2})$ to $X_1$ factoring through $X_T'$. Since $X$ does not belong to $Gen(\Xcal_{A_{\Sigma_T}})$, the $A$-module $X_T'$ cannot surject onto $X$. Thus, we can replace $X_T$ by $X_T'$ and repeat the whole argument with different $X_1$ and $X_2$, until (I) and (II) are fulfilled. Regarding (III), we will first show that $Hom_A(\sigma_1^{X_2},\Xcal_{A_{\Sigma_T}})=0$. Suppose there is an indecomposable $A$-module $X_T'$ in $\Xcal_{A_{\Sigma_T}}$ and a map $f:P_1^{X_2}\rightarrow X_T'$ not factoring through $ker(\pi^{X_2})$. Consequently, we get a morphism 
$$h:X_T'\rightarrow X_T$$ 
such that $X_2$ surjects onto the cokernel of $h$, which belongs to $\Xcal_{A_{\Sigma_T}}$. This yields a contradiction. Using that $Hom_A(\sigma_1^{X_2},\Xcal_{A_{\Sigma_T}})=0$ and the fact that $Hom_A(\sigma_1^X,\Xcal_{A_{\Sigma_T}})=0$ as well as $Ext_A^1(X,\Xcal_{A_{\Sigma_T}})=0$, since $X$ lies in $^*\!\Xcal_{A_{\Sigma_T}}$, we can also conclude that $Hom_A(\sigma_1^{X_1},\Xcal_{A_{\Sigma_T}})=0$. 

Altogether, we have seen that we can choose the $A$-module $X_T$ in $\Xcal_{A_{\Sigma_T}}$ in a way such that $X_1$ and $X_2$ belong to $^*\!\Xcal_{A_{\Sigma_T}}$.
Since, by construction, $X_1$ and $X_2$ also belong to $Gen(\Xcal_{A_{\Sigma_T}})$, they belong to $\Sigma_T'$, by Proposition \ref{propsilt}. Consequently, keeping in mind diagram (\ref{diagram2}), we conclude that $X$ lies in $^*\!\Xcal_{A_{\Sigma_T'}}$. This finishes the proof.
\end{proof}

In particular, if $T$ is a $\tau$-tilting $A$-module, then $A_{\Sigma_T}$ is just given by localising at the set of indecomposable non split-projective $A$-modules in $addT$.
Let us finish with an example illustrating the previous result.

\begin{example}
Let $A$ be the self-injective $\mathbb{K}$-algebra $\tilde{A}_3^3$. By Theorem \ref{Naktilt} and Corollary \ref{cor pic}, we know that
$$|s\tau\mbox{-}tilt(A)|=|uniloc(A)|=\left(\begin{array}{c} 6 \\ 3 \end{array}\right)=20.$$ 
If we restrict ourselves to proper $\tau$-tilting $A$-modules, by Theorem \ref{Naktilt}(2) and the classification of the universal localisations of $A$, we get that 
$$|\tau\mbox{-}tilt(A)|=|uniloc^p(A)|=10.$$ 
The following table lists the $\tau$-tilting modules and their associated universal localisations, indicated by $\Sigma_T'$. \vspace*{0.4cm}\center
\begin{tabular}{|c|c|}
$\bf{\tau\mbox{-}tilt(A)}$ & $\bf{uniloc^p(A)}$ \\ \hline $T:=A=P_1\oplus P_2\oplus P_3$ & $\Sigma_T'=\{0\}$ \\ $T:=P_1\oplus P_3\oplus S_1$ & $\Sigma_T'=\{S_1\}$ \\ $T:=P_1\oplus P_2\oplus S_2$ & $\Sigma_T'=\{S_2\}$ \\ $T:=P_2\oplus P_3\oplus S_3$ & $\Sigma_T'=\{S_3\}$ \\ $T:=P_1\oplus P_1/rad^2P_1\oplus S_1$ & $\Sigma_T'=\{S_1,P_1/rad^2P_1\}$ \\ $T:=P_2\oplus P_2/rad^2P_2\oplus S_2$ & $\Sigma_T'=\{S_2,P_2/rad^2P_2\}$ \\ $T:=P_3\oplus P_3/rad^2P_3\oplus S_3$ & $\Sigma_T'=\{S_3,P_3/rad^2P_3\}$ \\ $T:=P_1\oplus P_1/rad^2P_1\oplus S_2$ & $\Sigma_T'=\{P_1/rad^2P_1\}$ \\ $T:=P_2\oplus P_2/rad^2P_2\oplus S_3$ & $\Sigma_T'=\{P_2/rad^2P_2\}$ \\ $T:=P_3\oplus P_3/rad^2P_3\oplus S_1$ & $\Sigma_T'=\{P_3/rad^2P_3\}$ \\
\end{tabular}
\end{example}

\end{document}